\documentclass[A4]{amsart}
\input xypic
\input xy \xyoption{all}
\usepackage{MnSymbol}
\usepackage{bbm}
\usepackage{amsmath}
\usepackage{graphicx}

\newtheorem{theorem}{Theorem}[section]
\newtheorem{proposition}[theorem]{Proposition}
\newtheorem{lemma}[theorem]{Lemma}
\newtheorem{corollary}[theorem]{Corollary}
\newtheorem*{thm:local}{Theorem \ref{TMainStatementPrimes}}

\newtheorem*{thm:htpyGolod}{Propoition \ref{PNecessary}}

\newtheorem*{thm:main}{Theorem \ref{TMainStatement2}}

\theoremstyle{definition}
\newtheorem{definition}[theorem]{Definition}

\newtheorem{conjecture}[theorem]{Conjecture}
\newtheorem{example}[theorem]{Example}
\newtheorem{question}[theorem]{Question}

\newcounter{bean}

\newcommand{\map}[3]{\ensuremath{#1\colon#2
 {\longrightarrow}#3}}
\newcommand{\seqm}[3]{\ensuremath{#1\stackrel{#2}
 {\longrightarrow}#3}}
\newcommand{\seqmm}[5]{\ensuremath{#1\stackrel{#2}
 {\longrightarrow}#3\stackrel{#4}{\longrightarrow}#5}}
\newcommand{\seqmmm}[7]{\ensuremath{#1\stackrel{#2}
 {\longrightarrow}#3\stackrel{#4}{\longrightarrow}#5
  \stackrel{#6}{\longrightarrow}#7}}

\newcommand{\floor}[1]{\ensuremath{\left\lfloor #1 \right\rfloor}}

\newcommand{\paren}[1]{\ensuremath{\left( #1 \right)}}
\newcommand{\br}[1]{\ensuremath{\left\{ #1 \right\}}}

\newcommand{\abs}[1]{\ensuremath{\left|#1\right|}}

\newcommand{\cwedge}[3]{\displaystyle\bigwedge^{#2}_{#1}#3}
\newcommand{\cvee}[3]{\displaystyle\bigvee^{#2}_{#1}#3}
\newcommand{\cprod}[3]{\displaystyle\prod^{#2}_{#1}#3}

\newcommand{\cplus}[3]{\displaystyle\bigoplus^{#2}_{#1}#3}
\newcommand{\cdunion}[3]{\displaystyle\coprod^{#2}_{#1}#3}
\newcommand{\cunion}[3]{\displaystyle\bigcup^{#2}_{#1}#3}
\newcommand{\csum}[3]{\displaystyle\sum^{#2}_{#1}#3}
\newcommand{\cunionmulti}[4]{\displaystyle\bigcup^{#3}
_{\renewcommand{\arraystretch}{0.6}\begin{matrix}\scriptstyle #1 \cr \scriptstyle #2\end{matrix}\renewcommand{\arraystretch}{1.2}}#4}
\newcommand{\csummulti}[4]{\displaystyle\sum^{#3}
_{\renewcommand{\arraystretch}{0.6}\begin{matrix}\scriptstyle #1 \cr \scriptstyle #2\end{matrix}\renewcommand{\arraystretch}{1.2}}#4}

\newcommand{\cset}[2]{\br{#1\,\,\middle\vert\,\,#2}}

\newcommand{\qqed}{\hfill\square}

\renewcommand{\k}{\mathbf k}

\newcommand{\ul}[1]{\ensuremath{\underline{#1}}}
\newcommand{\mc}[1]{\ensuremath{\mathcal{#1}}}
\newcommand{\mb}[1]{\ensuremath{\mathbb{#1}}}

\newcommand{\Z}{\ensuremath{\mathbb{Z}}}

\newcommand{\ID}{\ensuremath{\mathbbm{1}}}

\renewcommand{\map}{\ensuremath{\mbox{map}}}
\newcommand{\bd}{\ensuremath{\partial}}

\newcommand{\wcolon}{\ensuremath{\,\colon\,}}

\begin{document}
\title{Configuration Spaces and Polyhedral Products}

\author{Piotr Beben}
\address{\scriptsize{School of Mathematics, University of Southampton,Southampton SO17 1BJ, United Kingdom}} 
\email{P.D.Beben@soton.ac.uk} 
\author{Jelena Grbi\'c} 
\address{\scriptsize{School of Mathematics, University of Southampton,Southampton SO17 1BJ, United Kingdom}}  
\email{J.Grbic@soton.ac.uk} 

\thanks{Research supported in part by The Leverhulme Trust Research Project Grant  RPG-2012-560.}

\begin{abstract}
This paper aims to find the most general combinatorial conditions under which a moment-angle complex $(D^2,S^1)^K$ is a co-$H$-space,
thus splitting unstably in terms of its full subcomplexes. In this way we study to which extent the conjecture holds that a moment-angle complex over a Golod simplicial complex is a co-$H$-space.
Our main tool is a certain generalisation of the theory of labelled configuration spaces. 
\end{abstract}

\maketitle

\section{Introduction}

Polyhedral products have been the subject of quite a bit of interest recently, 
beginning with their appearance as homotopy theoretical generalisations of various objects studied in toric topology. 
Of particular importance are the polyhedral products $(D^2,S^1)^K$ and $(\mb CP^\infty,\ast)^K$, 
known as \emph{moment-angle complexes} and \emph{Davis-Januszkiewicz spaces} respectively.
The homotopy theory of these spaces has many applications - from complex and symplectic geometry (c.f.~\cite{MR2285318, MR2431667, MR3180483}), 
to combinatorial and homological algebra (c.f.~\cite{MR1453579,MR846439}).
For example, moment-angle manifolds appear as intersection of quadrics or as quasitoric manifolds after taking a certain orbit space, 
\emph{Stanley-Reisner rings} of simplicial complexes are realised by the cohomology of Davis-Januszkiewicz spaces 
(equivalently, the equivariant cohomology ring of moment-angle complexes),
while the cohomology ring of moment-angle complexes is closely related to the study of the cohomology of local rings
(c.f.~\cite{MR1104531, MR1897064}). 
One would like to know how the combinatorics of the underlying simplicial complex $K$ encodes geometrical and topological properties of polyhedral products, and vice-versa. 
The case of Golod complexes is especially relevant. 
A ring $R=\k[v_1,\ldots,v_n]/I$ for $I$ a homogeneous ideal is said to be \emph{Golod} if all products and higher Massey products in 
$\mathrm{Tor}^+_{\k[v_1,\ldots,v_n]}(R,\k)$ vanish. 
Golod~\cite{MR0138667} showed that the Poincar\'e series of the homology ring $\mathrm{Tor}_{R}(\k,\k)$ of $R$ is a rational function 
whenever $R$ is Golod. From the context of combinatorics, 
a simplicial complex $K$ on vertex set $[n]=\{1,\ldots,n\}$ is said to be \emph{Golod} over $\k$ if the \emph{Stanley-Reisner ring} $\k[K]$ is Golod,
and if this is true for all fields $\k$ and $\k=\Z$, we simply say that $K$ is Golod.
Fixing $\k$ to be a field or $\mb Z$, by~\cite{MR1897064,MR2255969,MR2117435,MR0441987} there are isomorphisms of graded commutative algebras
$$
H^*((D^2,S^1)^K;\k)\cong \mathrm{Tor}_{\k[v_1,\ldots,v_n]}(\k[K],\k) \cong \cplus{I\subseteq[n]}{}{\tilde H^*(\Sigma^{|I|+1}|K_I|;\k)}
$$
where $\k[K]$ is the Stanley-Reisner ring of $K$, $K_I$ is the restriction of $K$ to vertex set $I\subseteq[n]$,
and the multiplication in the rightmost algebra is realised by maps
$$ 
\iota_{I,J}\wcolon\seqm{|K_{I\cup J}|}{}{|K_I\ast K_J|\cong |K_I|\ast |K_J|\simeq\Sigma |K_I|\wedge |K_J|}, 
$$
induced by the canonical inclusions \seqm{K_{I\cup J}}{}{K_I\ast K_J} whenever $I$ and $J$ are non-empty and disjoint. 
The Golod condition can then be reinterpreted as $\iota_{I,J}$ inducing trivial maps on $\k$-cohomology for disjoint non-empty 
$I$ and $J$ together with Massey products vanishing in $H^+((D^2,S^1)^K;\k)$.
This topological interpretation of the Golod condition has been a starting point for applying the homotopy theory of moment-angle complexes 
to the problem of determining which simplicial complexes $K$ are Golod, see~\cite{MR2321037} for example.
In the opposite direction, the cohomology of a moment-angle complex $(D^2,S^1)^K$ takes its simplest algebraic form when we restrict to Golod $K$.
Golod complexes are therefore a natural starting point for studying the homotopy types of moment-angle complexes.

Considerable work has been done on the homotopy theory of moment-angle complexes over Golod complexes
~\cite{MR2138475,MR2321037,MR3084441,MR3084442,arXiv:1306.6221,arXiv:1211.0873},
culminating in the following conjectured topological characterisation of the Golod complexes.
\begin{conjecture}
\label{CGolodConj}
A moment-angle complex $(D^2,S^1)^K$ is a co-$H$-space if and only if $K$ is Golod.
\end{conjecture}
The right-hand implication follows from the general fact that cohomology rings of co-$H$-spaces have trivial cup and Massey products. 
The left-hand implication, if true (or to whichever extent it is true), 
immediately gives a nice explicit combinatorial description for the homotopy type of $(D^2,S^1)^K$.
Namely, $(D^2,S^1)^K$ is a co-$H$-space if and only if 
$$
(D^2,S^1)^K\simeq \cvee{I\subseteq[n]}{}{\Sigma^{|I|+1}|K_I|}.
$$
This was shown by Iriye and Kishimoto~\cite{arXiv:1306.6221},
and follows from a general suspension splitting due to Bahri, Bendersky, Cohen, and Gitler 
(the \emph{BBCG splitting})~\cite{MR2673742}
\begin{equation}
\label{ESplittingBBCG}
\Sigma(D^2,S^1)^K\simeq \cvee{I\subseteq[n]}{}{\Sigma^{|I|+2}|K_I|}.
\end{equation}

Our goal in this paper is to determine the extent to which Conjecture~\ref{CGolodConj} holds.
To do this we consider several simplifications, the first of these being localisation. 
We confirm a \emph{large primes} version of Conjecture~\ref{CGolodConj},
thus reproducing the rational result of Berglund~\cite{BerglundRational} 
without making any significant use of rational homotopy theory.

\begin{thm:local}
Localised at any sufficiently large prime $p$, 
$(D^2,S^1)^K$ is a co-$H$-space if and only if $K$ is Golod over $\mb Z_p$.
\end{thm:local} 

Another simplification is to strengthen the hypothesis on a simplicial complex being Golod.
One way of doing so is to require some suspension of each inclusion $\iota_{I,J}$ to be nullhomotopic 
(instead of only inducing trivial maps on cohomology) for every disjoint and non-empty $I,J\subsetneq [n]$.
In Proposition~\ref{PNecessary} we show that after appropriate suspensions of $\iota_{I,J}$, 
this is a necessary condition for $(D^2,S^1)^K$ to be a co-$H$-space, and we call this the \emph{homotopy Golod} condition.

\begin{thm:htpyGolod}
If $(D^2,S^1)^K$ is a co-$H$-space, then it is homotopy Golod. 
\end{thm:htpyGolod}

Moreover, we will see that under certain additional coherence conditions on these nullhomotopies, 
the homotopy Golod condition becomes sufficient for $(D^2,S^1)^K$ to be a co-$H$-space, 
though possible no longer necessary. We outline this coherence condition as follows. 
Take the natural inclusion $\iota_{I,J}\colon\seqm{|K_{I\cup J}|}{}{|K_I\ast K_J|}$ for disjoint $I$, $J$. When $I$, $J_1$, and $J_2$ are disjoint and non-empty and $\iota_I$ is the identity \seqm{|K_I|}{}{|K_I|}, there is a commutative diagram
\[\diagram
& |K_{I\cup J_1}\ast K_{J_2}|\drto^{(\iota_{I,J_1})\ast\iota_{J_2}}\\
|K_{I\cup J_1\cup J_2}|\urto^{\iota_{I\cup J_1,J_2}}\drto^{\iota_{I,J_1\cup J_2}} && |K_I\ast K_{J_1}\ast K_{J_2}|\\
& |K_{I}\ast K_{J_1\cup J_2}|\urto^{\iota_{I}\ast(\iota_{J_1,J_2})}.
\enddiagram\]
Then in defining the coherent homotopy Golod condition, we require the composites
$$
\seqmm{Cone(\Sigma|K_{I\cup J_1\cup J_2}|)}{\hat\iota_{I\cup J_1,J_2}}{\Sigma|K_{I\cup J_1}\ast K_{J_2}|}{\Sigma(\iota_{I,J_1})\ast\iota_{J_2}}{\Sigma|K_I\ast K_{J_1}\ast K_{J_2}|}
$$
$$
\seqmm{Cone(\Sigma|K_{I\cup J_1\cup J_2}|)}{\hat\iota_{I,J_1\cup J_2}}{\Sigma|K_I\ast K_{J_1\cup J_2}|}{\Sigma\iota_{I}\ast(\iota_{J_1,J_2})}{\Sigma|K_I\ast K_{J_1}\ast K_{J_2}|}
$$
to be homotopic to each other via a homotopy that is fixed on the base $\Sigma|K_{I\cup J_1\cup J_2}|$ of the cone $Cone(\Sigma|K_{I\cup J_1\cup J_2}|)$,
where $\hat\iota_{I\cup J_1,J_2}$ and $\hat\iota_{I,J_1\cup J_2}$ are the extensions given by the nullhomotopies
of $\Sigma\iota_{I\cup J_1,J_2}$ and $\Sigma\iota_{I,J_1\cup J_2}$. One then continues in this manner defining higher coherences for longer joins. 
This bears some resemblance to the coherence of homotopy in Stasheff's higher homotopy associativity of $A_n$-spaces~\cite{MR0158400}. 
In our case, the \emph{associahedron} is replaced with the simplicial complex associated with ordered partitions of the vertex set $[n]$
(in other words, the dual of an order $n$ \emph{permutohedron}), 
while $H$-space multiplication maps are replaced with nullhomotopies of certain reduced diagonal maps, 
which up to homeomorphism are the above inclusions of cones into joins of full subcomplexes. 
This is made precise in Section~\ref{SGolod}, which we call the \emph{coherent homotopy Golod} condition. 
The left-hand implication of the conjecture then holds for coherently homotopy Golod complexes.

\begin{thm:main}
If $K$ is coherently homotopy Golod, then $(D^2,S^1)^K$ is a co-$H$-space.  
In particular, coherently homotopy Golod complexes are Golod.
\end{thm:main}

In particular, notice that the coherent homotopy Golod condition implies that all Massey products vanish in the cohomology of our moment-angle complex. 
The \emph{extractible} complexes defined in~\cite{arXiv:1306.6221} are currently the most general subclass of Golod complexes $K$
for which $(D^2,S^1)^K$ is known to be a co-$H$-space. We will show that extractible complexes are coherently homotopy Golod, 
and give an example of a class of Golod complexes that are coherently homotopy Golod, but which are not extractible.
In a follow-up paper~\cite{BebenGrbic2} we use a weaker form of the coherent homotopy Golod condition to show that the much simpler homotopy Golod condition
is both sufficient and necessary for $(D^2,S^1)^K$ to be a co-$H$-space whenever $K$ is $\frac{n}{3}$-neighbourly.

The main idea in this paper is the following well-known fact due to Ganea:
a space $Y$ is a co-$H$-space if and only if the evaluation map \seqm{\Sigma\Omega Y}{ev}{Y} has a right homotopy inverse.
Our method for constructing right homotopy inverses of \seqm{\Sigma\Omega (D^2,S^1)^K }{ev}{(D^2,S^1)^K} depends on a construction of a certain \emph{scanning map}
$$
\gamma\wcolon\seqm{\mc C((D^1,S^0)^K)}{}{\Omega (D^2,S^1)^K}
$$
together with a suspension splitting $\Sigma\mc C((D^1,S^0)^K)\simeq \bigvee_i \Sigma\mc D_i((D^1,S^0)^K)$.
Here $\mc C((D^1,S^0)^K)$ is a configuration space that generalises classical labelled configuration spaces by allowing particles to collide under certain rules,
and $\mc D_i((D^1,S^0)^K)$ is a certain \emph{quotiented} configuration space.
We remark that there are various constructions of configuration spaces in the literature that are related to these~\cite{2009arXiv0901.2871D,MR1820902,MR1851264,MR2255967,MR2492776}.
The upshot of this one is that the summands $\Sigma\mc D_i((D^1,S^0)^K)$ in the splitting of $\Sigma\Omega(D^2,S^1)^K$ 
are in a sense more combinatorial in their behaviour in comparison to polyhedral products,
allowing us to obtain combinatorial statements about the homotopy type of $(D^2,S^1)^K$. 
This is analogous to the summands in the BBCG splitting~\eqref{ESplittingBBCG} being combinatorial objects
- each of them is known to be homeomorphic to certain \emph{quotiented} moment-angle complexes.

Throughout this paper, $-1$ is taken to be the basepoint of $D^1=[-1,1]$.
Points in $S^1$ are thought of as real numbers in $D^1$ with $-1$ and $1$ identified as the basepoint. 
The suspension $\Sigma X$ of a space $X$ is taken to be the reduced suspension 
$D^1\times X/(\{-1,1\}\times X\cup D^1\times\{\ast\})$ whenever $X$ is basepointed with basepoint $\ast$. 
Otherwise it is the unreduced suspension, in other words, the quotient space of $D^1\times X$ under identifications
$(-1,x)\sim \ast_{-1}$ and $(1,x)\sim \ast_{1}$. 
In any case, $\Sigma X$ is always basepointed, in the unreduced case the basepoint is taken to be $\ast_{-1}$.

Fix a product of connected basepointed spaces $X=X_1\times\cdots\times X_n$ and a subspace $W\subseteq X$.
For any product $Y=Y_1\times\cdots\times Y_n$ of basepointed spaces, 
define the \emph{coordinate smash} $Y\wedge_{X} W$ of $Y$ and $W$ in $X$ to be the subspace of 
$Y\wedge_{X} X=(Y_1\wedge X_1)\times\cdots\times (Y_n\wedge X_n)$ given by
$$
Y\wedge_{X} W=\cset{((y_1,x_1),\ldots,(y_n,x_n))\in Y\wedge_{X} X}{(x_1,\ldots,x_n)\in W}.
$$
In particular, when each $Y_i$ is an $\ell$-sphere $S^\ell$,
we call $Y\wedge_{X} W$ the $\ell$-\emph{fold coordinate suspension of} $W$ in $X$ 
and denote it by $W^\ell$. This is then the subspace of 
$
X^\ell=(\Sigma^\ell X_1)\times\cdots\times (\Sigma^\ell X_n)
$ 
given by
$$
W^\ell=\cset{((t_1,x_1),\ldots,(t_n,x_n))\in X^\ell}{(x_1,\ldots,x_n)\in W}
$$
for pairs $(t_i,x_i)$ in the reduced suspension 
$\Sigma^\ell X_i=D^\ell\times X_i/(\bd D^\ell\times X_i\cup D^\ell\times\{\ast\})$, $D^\ell$ being the unit $\ell$-disk. 
Notice $W^0=W$ and $(W^{\ell_1})^{\ell_2}=W^{\ell_1+\ell_2}$.

\begin{example}
Let $W$ be some union of product subspaces $A_1\times\cdots\times A_n$. 
Then $W^\ell$ is a union of the subspaces $\Sigma^\ell A_1\times\cdots\times \Sigma^\ell A_n$.
In particular, thinking of the $A_i$'s as cells in $X_i$, if $W$ is a subcomplex of $X$, 
then we see that $W^\ell$ is a subcomplex of $X^\ell$.
\end{example}

\begin{example}
\label{EPolyhedral}
If $(\ul X,\ul A)=((X_1,A_1),\ldots,(X_n,A_n))$ is a sequence of pairs of spaces,  
$K$ is a simplicial complex on $n$ vertices, and $W$ is the \emph{polyhedral product} 
$$
(\ul X,\ul A)^K=\cunion{\sigma\in K}{}{Y^\sigma_1\times\cdots\times Y^\sigma_n}\quad\subseteq\quad X_1\times\cdots\times X_n,
$$
where
$$
Y^\sigma_i=
\begin{cases}
X_i & \mbox{if }i\in\sigma\\
A_i & \mbox{if }i\nin\sigma,
\end{cases}
$$
then $W^\ell$ is the polyhedral product $(\Sigma^\ell \ul X,\Sigma^\ell \ul A)^K\subseteq X^\ell$ of the $n$-tuple
$(\Sigma^\ell \ul X,\Sigma^\ell \ul A)$ of pairs of spaces
$((\Sigma^\ell X_1,\Sigma^\ell A_1),\ldots,(\Sigma^\ell X_n,\Sigma^\ell A_n))$.  
\end{example}

\subsection{Cofibrations and Splittings}

Given any subset $I=\{i_1,\ldots,i_k\}\subseteq [n]= \{1,\ldots,n\}$,
let $W^\ell_I$ be the image of $W^\ell$ under the projection \seqm{X^\ell}{}{X^\ell_I},
and $\hat W^\ell_I$ the image of $W^\ell_I$ under the quotient map \seqm{X^\ell_I}{}{\hat X^\ell_I}, 
where
$$
X^\ell_I=\Sigma^\ell X_{i_1}\times\cdots\times \Sigma^\ell X_{i_k}, 
$$
$$
\hat X^\ell_I=\Sigma^\ell X_{i_1}\wedge\cdots\wedge\Sigma^\ell X_{i_k}.
$$

\begin{definition}
We say $W$ satisfies the \emph{simplicial property} if given any point $x\in W$, 
the point $x'\in X$ obtained from $x$ by replacing some coordinate in $x$ with a basepoint is also a point in $W$.
\end{definition}

\begin{example}
Polyhedral products have the simplicial property.
If $\ell\geq 1$, then $W^\ell$ always has the simplicial property as a subspace of $X^\ell$. 
Also, if each $X_i$ is a basepointed $CW$-complex (i.e. with a fixed basepoint,
built up by attaching cells of dimension $\geq 1$ via basepoint preserving attaching maps), 
and the zero skeleton of $X_i$ contains a single point (the basepoint), then $W$ has the simplicial property. 
\end{example}

Suppose that $\ell>0$, or $\ell=0$ and $W$ satisfies the simplicial property. Then  
$$
W^\ell_I \cong \cset{(x_1,\ldots,x_n)\in W^\ell}{x_i=\ast\mbox{ if }i\nin I}.
$$
We will often think of $W^\ell_I$ as this subspace of $W^\ell$.
Thus, for any $0\leq k\leq n$, we can take the union
\begin{equation}
\label{W_k}
W^\ell_k=\cunion{I\subseteq [n],|I|=k}{}{W^\ell_I}.
\end{equation}
In other words, 
$$
W^\ell_k = \cset{(x_1,\ldots,x_n)\in W^\ell}{\mbox{at least }n-k\mbox{ coordinates } x_i\mbox{ are the basepoint }\ast\in X_i}.
$$
The equality $(W^\ell)_k=(W_k)^\ell$ is clear, so we forget brackets. Take the quotient spaces
$$
\hat W^\ell_k = \frac{W^\ell_k}{W^\ell_{k-1}}. 
$$
and $\hat W^\ell_I=W^\ell_I/(W^\ell_I)_{|I|-1}$.
Applying these constructions to $X$ in place of $W$, 
since $X^\ell_I\cap X^\ell_J=X^\ell_{I\cap J}$ as subspaces of $X^\ell$,  
and $X^\ell_I\cap X^\ell_{k-1}=(X^\ell_I)_{|I|-1}$ if $|I|=k$, then
\begin{equation}
\label{EHomeoSplitting}
\hat X^\ell_k=\frac{X^\ell_k}{X^\ell_{k-1}}=\cvee{I\subseteq [n],|I|=k}{}{\hat X^\ell_I}
\end{equation}
(note not to confuse $X^\ell_k$ with the factors $X_i$ of $X$).
One can think of $\hat W^\ell_k$ as the image of $W^\ell_k$ under the quotient map \seqm{X^\ell_k}{}{\hat X^\ell_k},
and this restricts to \seqm{W^\ell_I}{}{\hat W^\ell_I} for any $|I|=k$. Thus
\begin{equation}
\label{EHomeoSplittingW}
\hat W^\ell_k=\cvee{I\subseteq[n],|I|=k}{}{\hat W^\ell_I}.
\end{equation}
Mapping $\hat W^\ell_I$ via the homeomorphism \seqm{\hat X^\ell_I}{\cong}{\Sigma^{\ell |I|}\hat X_I}
that sends $((t_1,x_1),\ldots,(t_{|I|},x_{|I|}))$ to the point $((t_1,\ldots,t_{|I|}),(x_1,\ldots,x_{|I|}))$,
we see that 
$$
\hat W^\ell_I\cong \Sigma^{\ell |I|}\hat W_I
$$
$$\hat W^\ell_k\cong\Sigma^{\ell k}\hat W_k.
$$
Note $W^\ell_I$ and $W^\ell_k$ are subcomplexes of $W^\ell$ if each $X_i$ is a $CW$-complex, and $W$ is a subcomplex of $X$.

For the remainder of this section we assume that each $X_i$ is a connected basepointed $CW$-complex, 
and $W$ is a connected subcomplex of their product $X$ satisfying the simplicial property. 
The next two propositions generalise the polyhedral product splittings in~\cite{MR2673742} and~\cite{arXiv:1306.6221}.
For example, we may recover the BBCG suspension splitting of $W^1=(D^2,S^1)^K$ by applying the
homeomorphism $\hat W_I\cong \Sigma |K_I|$ given in~\cite{MR2673742}.

\begin{proposition}
\label{PSuspensionSplitting}
For $1\leq k\leq n$ and either $\ell> 0$, or $\ell=0$ and $W$ satisfying the simplicial property, there exists a suspension splitting
$$
\Sigma W^\ell_k \simeq \cvee{1\leq j\leq k}{}{\Sigma\hat W^\ell_j}
\quad\cong\quad\cvee{I\subseteq[n],|I|\leq k}{}{\Sigma\hat W^\ell_I}
\quad\cong\quad\cvee{I\subseteq[n],|I|\leq k}{}{\Sigma^{\ell|I|+1}\hat W_I}.
$$
\end{proposition}

\begin{proof}
Since $W^\ell_k$ has the simplicial property, given a subset $I=\{i_1,\ldots,i_s\}\subseteq [n]$, we can define a map
$q_I\wcolon\seqmm{W^\ell_k}{}{W^\ell_I}{}{\hat W^\ell_I\cong\Sigma^{\ell|I|}\hat W_I}$  
as the restriction of the composite \seqmm{X^\ell}{}{X^\ell_I}{}{\hat X^\ell_I} 
of the projection and quotient map to $W^\ell_k$. 
Notice that $q_I$ is the composite \seqmm{W^\ell_k}{}{\hat W^\ell_k}{\hat q_I}{\hat W^\ell_I} when $|I|=k$, 
where $\hat q_I$ is the quotient onto the summand $\hat W^\ell_I$ with respect to homeomorphism~\eqref{EHomeoSplittingW}.

We proceed by induction on $k$. Assume that for some $m<k$ a co-$H$-space sum
$$
h_m\wcolon\seqmm{\Sigma W^\ell_m}{pinch}{\cvee{I\subseteq[n],|I|\leq m}{}{\Sigma W^\ell_m}}
{\vee \Sigma q_I}{\cvee{I\subseteq[n],|I|\leq m}{}{\Sigma\hat W^\ell_I}}
$$
is a homotopy equivalence. This is trivial when $m=1$. 
We assume the order in which these co-$H$-space sums $h_j$ are taken is such that the maps $\Sigma q_I$ are summed before 
the maps $\Sigma q_J$ when $|I|<|J|$. Then there is a homotopy commutative diagram of cofibration sequences
\begin{equation}
\label{DCofib}
\diagram
\Sigma W^\ell_m\rto^{}\dto^{h_m} & \Sigma W^\ell_{m+1}\rto^{}\dto^{h_{m+1}} & \Sigma\hat W^\ell_{m+1}\dto^{\hat h_{m+1}}\\
\cvee{I\subseteq [n],|I|\leq m}{}{\Sigma\hat W^\ell_I}\rto^{} &
\cvee{I\subseteq [n],|I|\leq m+1}{}{\Sigma\hat W^\ell_I}\rto^{} &
\cvee{I\subseteq[n],|I|=m+1}{}{\Sigma\hat W^\ell_I}
\enddiagram
\end{equation}
where the bottom horizontal maps are respectively the summand-wise inclusion and the quotient map onto the appropriate summands. 
The top maps are the suspended inclusion and quotient. 
Since for $|I|=m+1$ the composite of the first two maps in the sequence
$\seqmmm{X^\ell_m}{}{X^\ell_{m+1}}{}{\hat X^\ell_{m+1}}{}{\hat X^\ell_I}$ is the constant map
(where the last map is the quotient map with respect to splitting \eqref{EHomeoSplitting}),
its restriction \seqmm{W^\ell_m}{}{W^\ell_{m+1}}{q_I}{\hat W^\ell_I} is the constant map.
Then since the suspended inclusion \seqm{\Sigma W^\ell_m}{}{\Sigma W^\ell_{m+1}} is a co-$H$-map,
and $h_m$ and $h_{m+1}$ are the co-$H$-space sums of the maps $\Sigma q_I$ as defined above, 
we see that the left-hand square commutes up to homotopy.
The right-hand map $\hat h_{m+1}$ is taken to be a co-$H$-space sum of the composites
$\Sigma\hat q_I\colon\seqmm{\Sigma\hat W^\ell_{m+1}}{\cong}{\bigvee_{|I|=m+1}{}{\Sigma\hat W^\ell_I}}{}{\Sigma\hat W^\ell_I}$ over all $|I|=m+1$,
where the first map is the suspension of the homeomorphism~\eqref{EHomeoSplittingW}, and the last map the suspended quotient map.
The first map being a homeomorphism, we see $\hat h_{m+1}$ induces an isomorphism on homology, thus is a homotopy equivalence.
We assume the co-$H$-space sum $\hat h_{m+1}$ is taken in the same order as the last maps $\Sigma q_I$ for $|I|=m+1$ 
are taken in the co-$H$-space sum $h_{m+1}$. 
Then since $\Sigma q_I$ is the composite of the suspended quotient map \seqm{\Sigma W^\ell_{m+1}}{}{\Sigma\hat W^\ell_{m+1}} and $\Sigma\hat q_I$,
and the suspended quotient map being a co-$H$-map, we see that the right-hand square commutes up to homotopy.

The left-hand and right-hand vertical maps being homotopy equivalences, $h_{m+1}$ induces an isomorphism of homology groups. 
All spaces being simply connected $CW$-complexes, $h_{m+1}$ is a homotopy equivalence. This completes the induction.
\end{proof}

\begin{lemma}
\label{LSimplyConnected}
$W^\ell_k$ is simply connected for $\ell\geq 2$ and each $k\leq n$.
Moreover, if $W$ is connected, then $W^1_k$ is also simply connected.
\end{lemma}
\begin{proof}
Since $X^\ell$ has no $1$-cells when $\ell\geq 2$, neither does its subcomplex $W^\ell_k$, so it is simply-connected.

Suppose $W$ is connected. Then each $W_I$ is connected, and since $W$ has the simplicial property, 
the $W_I$'s intersect at the basepoint, so $W_k$ is also connected. 
Notice the $1$-skeleton of $\Sigma X_i$ consists of points $(t,y)$ such that $y$ is in the $0$-skeleton of $X_i$. 
Since cells in $X^1$ are products of cells in each $\Sigma X_i$, 
the $1$-skeleton of the subcomplex $W^1_k\subseteq X^1$ is contained in the subspace $V\subseteq W^1_k=(W_k)^1$ consisting of points
$z=((t_1,x_1),\ldots,(t_n,x_n))$ such that $x=(x_1,\ldots,x_n)$ is a point in the $0$-skeleton of $W_k$.
Notice that $V$ is a wedge of subspaces $V_x=\cset{((t_1,x_1),\ldots,(t_n,x_n))}{0\leq t_i\leq 1}$ for each of the disjoint points $x$ in the $0$-skeleton of $W_k$.
Since $W_k$ is connected, we can define a nullhomotopy of $V_x$ in $W^1_k$ by picking a path in $W^1_k$ from $x$ to the basepoint 
$(\ast,\ldots,\ast)$ to define a homotopy of $z=((t_1,x_1),\ldots,(t_n,x_n))$ to the basepoint $((t_1,\ast),\ldots,(t_n,\ast))\sim\ast$. 
This therefore defines a nullhomotopy of $V$ in $W^1_k$, and since the $1$-skeleton of $W^1_k$ is contained in $V$, a nullhomotopy of its $1$-skeleton inside $W^1_k$. Thus $W^1_k$ is simply connected.  
\end{proof}

The following folk statement will be useful, and for completeness we include a proof.

\begin{lemma}
\label{LCoMult}
If $Y$ is a co-$H$-space, and $\psi\colon\seqm{Y}{}{Y\vee Y}$ its comultiplication,
then $\Sigma\psi$ is homotopic to the pinch map $\seqm{\Sigma Y}{pinch}{\Sigma Y\vee \Sigma Y}$.
\end{lemma}

\begin{proof}
Let $H$ and $H'$ be the homotopies to the identity of the composites 
\seqmm{Y}{\psi}{Y\vee Y}{\ID\vee\ast}{Y} and \seqmm{Y}{\psi}{Y\vee Y}{\ast\vee\ID}{Y}.
The composite \seqmm{Y}{\psi}{Y\vee Y}{\iota}{Y\times Y}, where $\iota$ is the canonical inclusion, is homotopic to the diagonal map \seqm{Y}{\vartriangle}{Y\times Y}
via the homotopy $G_t(y)=(H_t(y),H'_t(y))$.  
Take the composite 
$$
f\wcolon\seqmm{\Sigma(Y\times Y)}{pinch}{\Sigma(Y\times Y)\vee \Sigma(Y\times Y)}{}{\Sigma Y\vee\Sigma Y}
$$
where the last map is the wedge sum of the suspended projection maps 
$\seqm{\Sigma(Y\times Y)}{\Sigma \pi_i}{\Sigma Y}$, $i=1,2$, 
projecting onto the first factor and the second factor respectively. 
The composite $f$ is a left homotopy inverse of $\Sigma\iota$. Therefore
$
\seqmmm{\Sigma Y}{\Sigma\psi}{\Sigma (Y\vee Y)}{\Sigma \iota}{\Sigma(Y\times Y)}{f}{\Sigma Y\vee\Sigma Y}
$
is homotopic to $\Sigma\psi$, and 
$
\seqmm{\Sigma Y}{\Sigma\vartriangle}{\Sigma (Y\times Y)}{f}{\Sigma Y\vee\Sigma Y}
$
is the pinch map \seqm{\Sigma Y}{pinch}{\Sigma Y\vee\Sigma Y}. These two composites are homotopic since $\Sigma\iota\circ\Sigma\psi=\Sigma(\iota\circ\psi)$ is homotopic to $\Sigma\vartriangle$.
\end{proof}

\begin{proposition}
\label{PSplitting}
Let $\ell\geq 0$ and $1\leq k\leq n$, and consider the following: 
\begin{itemize}
\item[(a)] $W^\ell_k$ is a co-$H$-space;
\item[(b)] there is a splitting
$$
W^\ell_k \simeq \cvee{0\leq j\leq k}{}{\hat W^\ell_j}
\cong\cvee{I\subseteq[n],|I|\leq k}{}{\hat W^\ell_I}
\cong\cvee{I\subseteq[n],|I|\leq k}{}{\Sigma^{\ell|I|}\hat W_I};
$$
\item[(c)] the quotient map \seqm{W^\ell_m}{}{\hat W^\ell_m} has a right homotopy inverse for each $m\leq k$;
\item[(d)] $W^\ell_I$ is a simply connected co-$H$-space for each $I\subseteq[n]$ with $|I|\leq k$;
\item[(e)] the quotient map \seqm{W^\ell_I}{}{\hat W^\ell_I} has a right homotopy inverse for each $I\subseteq[n]$ with $|I|\leq k$.
\end{itemize}
Then $(a)\Leftrightarrow(b)\Leftrightarrow(c)\Leftrightarrow(d)\Leftrightarrow(e)$ when $\ell\geq 2$, 
or when $\ell=1$ and $W$ is connected. 
Generally, $(a)\Rightarrow(b)$ provided $W^\ell_j$ is simply connected for $j\leq k$. 
\end{proposition}

\begin{proof}

When $\ell\geq 2$, or $\ell=1$ and $W$ is connected, 
Lemma~\ref{LSimplyConnected} ensures that all our spaces are simply connected.
The implication $(b)\Rightarrow(a)$ is clear when $\ell\geq 1$. 
To see that $(a)\Rightarrow(b)$ holds, for $m\leq k$, consider the composite
$$
h'_m\wcolon\seqmmm{W^\ell_m}{}{W^\ell_k}{\psi'}{\cvee{I\subseteq[n],|I|\leq m}{}{W^\ell_k}}{\vee q_I}
{\cvee{I\subseteq[n],|I|\leq m}{}{\hat W^\ell_I}},
$$
where the first map is the inclusion and $\psi'$ is given by iterating the comultiplication \seqm{W^\ell_k}{\psi}{W^\ell_k\vee W^\ell_k} 
in some order. 
By Lemma~\ref{LCoMult}, $\Sigma\psi'$ is homotopic to the iterated pinch map, so by naturality the composite of $\Sigma\psi'$
with the suspended inclusion \seqm{\Sigma W^\ell_m}{}{\Sigma W^\ell_k} is an iterated pinch map.
The map $\Sigma h'_m$ is then homotopic to the homotopy equivalence $h_m$ in the proof of Proposition~\ref{PSuspensionSplitting},
meaning $h'_m$ induces isomorphisms of homology groups, and since this is a map between simply-connected $CW$-complexes,
$h'_m$ is a homotopy equivalence. 

Suppose $(b)$ holds. From the commutativity of the right square in diagram~(\ref{DCofib}), 
we may take as a right homotopy inverse of the suspended quotient map 
\seqm{\Sigma W^\ell_m}{}{\Sigma \hat W^\ell_m} the composite
$$
f\wcolon\seqmmm{\Sigma \hat W^\ell_m}{\cong}{\cvee{I\subseteq[n],|I|=m}{}{\Sigma\hat W^\ell_I}}{}
{\cvee{I\subseteq[n],|I|\leq m}{}{\Sigma\hat W^\ell_I}}{h^{-1}_m}{\Sigma W^\ell_m},
$$
where the second map is the inclusion.
Since $h_m\simeq \Sigma h'_m$ and $h_m'$ is a homotopy equivalence, we may take $h^{-1}_m=\Sigma (h'_m)^{-1}$.
All the maps in this composite being suspensions, $f$ desuspends to a map $f'$ such that the composite 
\seqmm{\hat W^\ell_m}{f'}{W^\ell_{m}}{}{\hat W^\ell_m} induces isomorphisms on homology groups.
Since each of the summands in the splitting of $\hat W^\ell_m$ are summands in the splitting of $\Sigma W^\ell_m$, 
this last composite must be homotopic to the identity. Thus $(b)\Rightarrow(c)$ holds. 

Since each of the cofibrations \seqmm{W^\ell_{m-1}}{}{W^\ell_m}{}{\hat W^\ell_m} 
for $m\leq k$ trivializes when there is right homotopy inverse of \seqm{W^\ell_m}{}{\hat W^\ell_m}, $(c)\Rightarrow(b)$ holds.
Thus, we have shown $(a)\Leftrightarrow(b)\Leftrightarrow(c)$.

The restriction \seqm{W^\ell_k}{}{W^\ell_I} of the projection map 
\seqm{X^\ell}{}{X^\ell_I} to $W^\ell_k$ has a right inverse \seqm{W^\ell_I}{}{W^\ell_k} whenever $|I|\leq k$,
given by restricting the inclusion \seqm{X^\ell_I}{}{X^\ell} to $W^\ell_I$.  Thus $(a)\Rightarrow(d)$.

Suppose $(d)$ holds. By $(a)\Rightarrow (c)$  with $W^\ell_I$ in place of $W^\ell_k$, 
the quotient map \seqm{W^\ell_I}{}{\hat W^\ell_I} has a right homotopy inverse. Thus $(d)\Rightarrow(e)$.

Suppose $(e)$ holds. For any $|I|=m$ we have a commutative square 
\[\diagram
W^\ell_I\rto^{}\dto^{} & \hat W^\ell_I\dto^{} \\
W^\ell_m \rto^{} & \hat W^\ell_m \rto^-{\cong} & \cvee{J\subseteq[n],|J|=m}{}{\hat W^\ell_J}
\enddiagram\]
with the vertical maps the inclusions, the right-hand one being the inclusion into the summand $\hat W^\ell_I$. 
Therefore the wedge sum of maps \seqmm{\hat W^\ell_I}{s_I}{W^\ell_I}{}{W^\ell_m} over all $|I|=m$,
where $s_I$ is our given right homotopy inverse, is a right homotopy inverse of \seqm{W^\ell_m}{}{\hat W^\ell_m}. 
Thus $(e)\Rightarrow (c)$.

\end{proof}

Proposition~\ref{PSplitting} has an interesting interpretation for polyhedral products
$(\ul X,\ul A)^K$.

\begin{corollary}
\label{CSplitting}
If $(\ul X,\ul A)^K$ is a co-H-space, then 
$(\ul X ,\ul A)^K \simeq \cvee{I\subseteq [n]}{}{\widehat{(\ul X,\ul A)}^{K_I}}$.
\end{corollary}

Note that by~\cite{MR2673742}, when each $X_i$ is contractible, 
then $(\ul X ,\ul A)^K \simeq \cvee{I\not\in K}{}{|K_I|\ast \hat{A}^I}$,
and when each $A_i$ is contractible, 
$\widehat{(\ul X,\ul A)}^{K_I}\simeq\cwedge{i \in I}{}{X_i}$ when $K_I$ is a simplex, 
otherwise homotopy equivalent to a point.

\section{Configuration Spaces}

Let $M$ be any path connected space, $N\subseteq M$ a subspace, and $Y$ a basepointed space with basepoint~$\ast$. 
Let $D_0(M,N;Y)=\ast$ and take the quotient space
$$
D_k(M,N;Y) = \cdunion{i=0}{k}{M^{\times i}\times Y^{\times i}}/\sim
$$
where the equivalence relation $\sim$ is given by
$$
(z_1,\ldots,z_i;x_1,\ldots,x_i)\sim (z_{\sigma(1)},\ldots,z_{\sigma(i)};x_{\sigma(1)},\ldots,x_{\sigma(i)})
$$
for any permutation $\sigma$ in the symmetric group $\Sigma_i$, and 
$$
(z_1,\ldots,z_i;x_1,\ldots,x_i)\sim (z_1,\ldots,z_{i-1};x_1,\ldots,x_{i-1})
$$ 
whenever $x_i$ is the basepoint $\ast$ or $z_i\in N$. Then 
$$
D(M,N;Y) = \cunion{k=0}{\infty}{D_k(M,N;Y)}.
$$
Notice $D(M,N;Y)$ is just the infinite symmetric product $SP((M/N)\wedge Y)$.
We may think of $D(M,N;Y)$ as the space of (possibly colliding) particles in $M$ with labels in $Y$ that are annihilated in $N$. 
We thus refer to points in $D(M,N;Y)$ as \emph{configurations of particles} in $M$ with labels in $Y$, 
and by this we usually mean, without loss of generality, that the particles in such a configuration are \emph{non-degenerate}, 
or in other words, that they are outside $N$ and their labels are not the basepoint. 
The empty configuration is vacuously non-degenerate. 
It will sometimes be convenient to think of a configuration $(z_1,\ldots,z_i;x_1,\ldots,x_i)$ 
as a set $\{(z_1,x_1),\ldots (z_i,x_i)\}$ where each of the pairs $(z_i,x_i)$ are considered distinct
regardless of the values of $z_i$ or $x_i$.

 Fix $W$ to be a subspace of a product $X=X_1\times\cdots\times X_n$. 
\begin{definition}
A multiset of points $\mc S=\{\{x_1,\ldots,x_s\}\}\subset \coprod_i X_i$ is \emph{represented} by a point in $W$ 
if there exists a $(\bar x_1,\ldots,\bar x_n)\in W$ and some injective function
$f_{\mc S}\colon\seqm{[s]}{}{[n]}$ such that $x_i=\bar x_{f_{\mc S}(i)}$.

Equivalently, when $s>0$, 
$\mc S$ is represented by a point in $W$ if $(x_{\sigma(1)},\ldots,x_{\sigma(s)})$ is in $W_I$ 
for some distinct permutation $\sigma\in Aut([s])$ and $I\subseteq [n]$ with $|I|=s$.
\end{definition}

\begin{remark}
$\mc S$ is not represented by a point if it contains two or more elements in the same $X_i$ for some $i$. 
\end{remark}

\begin{definition}

Let $X^\vee=X_1\vee\cdots\vee X_n$ be the wedge at basepoints $\ast\in X_i$. 
Define $\mc C(M,N;W)$ to be the subspace of $D(M,N;X^\vee)$ given by the following rule:
a configuration of non-degenerate particles $(z_1,\ldots,z_k;x_1,\ldots,x_k)\in D(M,N;X^\vee)$  
is in $\mc C(M,N;W)$ if and only if 
\begin{itemize}
\item for each increasing integer sequence $1\leq i_1<\cdots<i_s\leq k$, if $z_{i_1}=\cdots=z_{i_s}$,
then the multiset of labels $\{\{x_{i_1},\ldots,x_{i_s}\}\}$ is represented by some point in $W$.
\end{itemize}
\end{definition}

The condition $z_{i_1}=\cdots=z_{i_s}$ is vacuously true when $s=1$, 
thus, the label $x$ of each particle is an element of $W_{\{i\}}\subseteq X_i$ for some $i$.
We filter $\mc C(M,N;W)$ by the subspaces
$$
\mc C_k(M,N;W)=\mc C(M,N;W)\cap D_k(M,N;X^\vee)
$$
having no more than $k$ non-degenerate particles.

These configuration spaces generalise McDuff's and B\"odgheimer's~\cite{MR0358766,MR922926},
originally studied by Boardman, Vogt, May, Giffen, and Segal.
Several other generalisations have appeared since then, some of these bearing certain similarities to these
(see~\cite{2009arXiv0901.2871D,MR1820902,MR1851264,MR2255967,MR2492776} for example).

\begin{example}
If $W=X_1\times\cdots\times X_n$, then
the space $\mc C(M,N;W)$ is homeomorphic to the product of classical labelled configuration spaces 
$C(M,N;X_1)\times\cdots\times C(M,N;X_n)$.
If $W=X_1\vee\cdots\vee X_n$, then $\mc C(M,N;W)=C(M,N;X_1\vee\cdots\vee X_n)$.
\end{example}

When referring to \emph{collided particles} in a labelled configuration of particles in $M$, 
we mean a subset of all particles in that configuration sharing the same point in $M$.
We can think of these configuration spaces as being like classical configuration spaces with labels in $X^\vee$,
but with certain collisions allowed governed by the way $W$ sits inside its ambient space $X$.
Here a single particle is considered to collide with itself, and at most $n$ particles can collide.

\subsection{Faithful Neighbourhoods}

Recall that a $CW$-complex $Z$ is said to be \emph{regular} if its attaching maps are embeddings.
A \emph{mapping cylinder neighbourhood} of a subspace $A\subseteq Z$ is a neighbourhood of $A$ in $Z$
homeomorphic to a mapping cylinder $M_f=A\sqcup([0,1]\times B)/\sim$, where $(0,b)\sim f(b)$, for some map $f\colon\seqm{B}{}{A}$.
An \emph{open mapping cylinder neighbourhood} is a subspace of $M_f$ of the form $A\sqcup ([0,1)\times B)/\sim$.
A \emph{sub-mapping cylinder} of $M_f$ is a subspace of $M_f$ of the form $A\sqcup ([0,1]\times B')/\sim$
for some subspace $B'\subseteq B$, where $f'$ restricts $f$ to $B'$.
We call the subspaces $(0,1]\times \{b\}\subseteq M_f$ for $b\in B$ the \emph{fibers} of $M_f$.

It is easy to see via a cell-by-cell induction that any subcomplex of a (geometric) simplicial complex has a mapping cylinder neighbourhood.
Moreover, each of its fibers lie entirely in the interior a single simplex. 
Since the cells of a regular $CW$-complex can be subdivided to give it the structure of a simplicial complex, 
subcomplexes of regular $CW$-complexes also have mapping cylinder neighbourhoods, with each of their fibers lying entirely in a single cell. 

The following is a coordinate-wise version of the neighbourhood deformation retraction property of a pair.

\begin{definition}
\label{DFaithful}
Any proper subcomplex $W$ has a \emph{faithful neighbourhood} if it has an open neighbourhood $U\subseteq X$ with 
associated basepoint preserving homotopies $h_{i,t}\colon\seqm{X_i}{}{X_i}$ for $i=1,\ldots,n$, such that: 
(1) $h_{i,0}$ is the identity and $h_{i,t}(\ast)=\ast$; 
(2) there is an open neighbourhood $U_i$ of the basepoint 
$\ast\in X_i$ such that $h_{i,1}(U_i)=\ast$, and $h_{i,t}(U_i)$ is an open neighbourhood of $\ast$ for each $t\in[0,1)$; and
(3) the homotopy $h_t = h_{1,t}\times\cdots\times h_{n,t}\colon\seqm{X}{}{X}$ satisfies 
$h_t(W)\subseteq W$ for each $t\in[0,1]$, $h_1(U)=W$, 
and $h_t(U)$ is an open neighbourhood of $W$ for each $t\in[0,1)$.
\end{definition}

Since any open subset of $X$ is a union of products of open subsets of each $X_i$, the projection \seqm{X}{}{X_I=\prod_{i\in I}X_i}
maps open subsets to open subsets for any $I\subseteq [n]$.
\begin{lemma}
\label{LProjection}
The image of our faithful neighbourhood $U$ of $W$ under the projection \seqm{X}{}{X_I}
is a faithful neighbourhood of $W_I$ in $X_I$, inheriting from $U$ the associated homotopies $h_{i,t}$ where $i\in I$.~$\qqed$
\end{lemma}

\begin{example}
\label{EFaithful2}
Given a simplicial complex $K$ on vertex set $[n]$, 
the real moment-angle complex $W=(D^1,S^0)^K$ with basepoint 0 has as a faithful neighbourhood the polyhedral product
$U=(D^1,[0,\frac{1}{4})\cup(\frac{3}{4},1])^K$ with each $h_{i,t}\colon\seqm{D^1}{}{D^1}$ 
defined explicitly by
$h_{i,t}(s)=s(1-t)$ for $s\in[0,\frac{1}{4}]$, 
$h_{i,t}(s)=s+t(s-\frac{1}{2})$ for $s\in(\frac{1}{4},\frac{3}{4})$, and
$h_{i,t}(s)=s+t(1-s)$ for $s\in[\frac{3}{4},1]$. 
Note $U$ and the $h_{i,t}$'s also form a faithful neighbourhood for each subcomplex $W_k$.
\end{example}

\begin{example}
\label{EFaithful}
More generally, let $Y$ be a regular $CW$-complex and $A$ a subcomplex. 
The polyhedral product $W=(Y,A)^K$ has as a faithful neighbourhood the polyhedral product $U=(Y,M_f)^K$, 
where $M_f$ is an open mapping cylinder neighbourhood $A\cup_f ([0,\frac{1}{2})\times B)$ of $A$ for some map $f\colon\seqm{B}{}{A}$.
Each homotopy $h_{i,t}\colon\seqm{Y}{}{Y}$ is given as follows.
Define $g_{i,t}\colon\seqm{Y}{}{Y}$ by extending to $Y$ the deformation retraction of $M_f$ onto $A$ to give our homotopy $g_{i,t}$.
Then define $g'_{i,t}\colon\seqm{Y}{}{Y}$ by taking an open mapping cylinder neighbourhood $M'_\ast$ of the basepoint $\ast$ in $A$,
extending it to an open mapping cylinder $U_i=M_\ast$ of $\ast$ in $Y$,
and extending the deformation retraction of $M_\ast$ onto $\ast$ to give our homotopy $g'_{i,t}$.
Note since $M'_\ast$ is a sub-mapping cylinder of $M_\ast$, the deformation retraction of $M_\ast$ onto $\ast$ restricts to a 
deformation retraction in $A$ of $M'_\ast$ onto $\ast$, so we can take $g'_{i,t}$ so that it restricts to $g'_{i,t}\colon\seqm{A}{}{A}$ on $A$,
which we do. Then let $h_{i,t}=g'_{i,t}\circ g_{i,t}$. 

If $A$ is a disjoint union $A'\coprod\{\ast\}$, so that $M_f$ is a disjoint union $M_{f'}\coprod M_\ast$, 
then each $W_k\subseteq Z_K(Y,A)$ has the subspace 
$$
U_k=\cset{(x_1,\ldots,x_n)\in Z_K(Y,M_f)}{x_i\in M_\ast\mbox{ for at least }n-k\mbox{ coordinates }x_i }
$$ 
as a faithful neighbourhood, with the same associated homotopies $h_{i,t}$.
\end{example}

\subsection{The Case of a Disk and its Boundary}

From now on we assume that each $X_i$ is a connected basepointed regular $CW$-complex, 
$W$ is a connected subcomplex of their product $X$ containing the basepoint, and $W$ has a faithful neighbourhood in $X$. 

Consider the inclusion 
$$
\iota\wcolon\seqm{W^\ell}{}{\mc C(D^\ell,\bd D^\ell;W)}
$$ 
given by sending 
$((z_1,x_1),\ldots,(z_n,x_n))$ to the configuration $(z_1,\ldots,z_n;x_1,\ldots,x_n)$.
Notice $\iota$ is the restriction of the likewise defined map \seqm{X^\ell}{\iota}{\mc C(D^\ell,\bd D^\ell;X)}.
This map in turn is just the product of inclusions \seqm{\Sigma^\ell X_i}{}{C(D^\ell,\bd D^\ell;X_i)}
into the classical configuration spaces, each of which is a homeomorphism of $\Sigma^\ell X_i$ onto its image~\cite{MR0358766,MR922926}. 
Therefore $\iota$ is homeomorphism of $W^\ell$ onto its image, 
and we think of $W^\ell$ as a subspace of $\mc C(D^\ell,\bd D^\ell;W)$ under this inclusion.
Notice that a configuration of non-degenerate particles is in the image of $\iota$ if and only if 
the multiset of labels of all its non-degenerate particles is represented by a point in $W$.

\begin{lemma}
\label{LBaseCase}
There is a deformation retraction of $\mc C(D^\ell,\bd D^\ell;W)$ onto the subspace $W^\ell$ given by the inclusion $\iota$. 
\end{lemma}

\begin{proof}
We use a two-step radial expansion argument, generalizing the deformation retraction for the 
classical labelled configuration space $C(D^\ell,\bd D^\ell;Y)$ onto $\Sigma^\ell Y$.
We think of $D^\ell$ as the closed unit-disk centered at the origin in $\mb R^\ell$, 
and the radial expansion $R_t\colon\seqm{D^\ell}{}{D^\ell}$ here is given by
$R_t(x)=tx$ if $\abs{tx}<1$ and $R_t(x)=\frac{tx}{\abs{tx}}$ if $\abs{tx}\geq 1$.

Let $\mc C$ denote $\mc C(D^\ell,\bd D^\ell;W)$, and $\mc C'$ be the subspace of
configurations in $\mc C$ whose non-degenerate particles have labels in distinct $X_i$'s 
(thus, there can be no more than $n$ of such particles). 
To define a deformation retraction of $\mc C$ onto $\mc C'$ we define a homotopy $H_t\colon\seqm{\mc C}{}{\mc C}$, 
for $t\in[0,\infty)$, by applying the radial expansion to non-degenerate particles in a given configuration $\omega\in \mc C$, 
while at the same time we homotope labels 'near' the basepoint towards the basepoint.  
We do this by picking a mapping cylinder neighbourhood $M_i$ of the basepoint in each $X_i$ 
(these exist since $X_i$ is a regular $CW$-complex), 
keeping a label fixed if it is the basepoint or outside the interior of some $M_i$, or else homotoping it towards the basepoint 
along the corresponding fibre $(0,1]\times\{b\}\subseteq M_i$ of the mapping cylinder, 
this being done at a speed that tends to zero the closer the label is to the boundary of $M_i$.
We may assume each of the fibers $(0,1]\times\{b\}$ of $M_i$ lie in a single cell in $X_i$. 
Then since $W$ as a subcomplex of $X$ is a union of products of cells in the $X_i$'s, 
under this homotopy of labels, collided particles still have their set of labels represented by a point in $W$. This defines $H$.

Notice only particles that are collided at the centre of $D^\ell$ are unaffected by the radial expansion, 
but then their multiset of labels is represented by a point in $W$, so in particular their labels must already be in distinct $X_i$'s. 
On the other hand, particles away from the centre may eventually enter the boundary $\bd D^\ell$ under the radial expansion, 
or have their labels homotoped to the basepoint, in either case becoming degenerate, 
until we are left with at most $n$ non-degenerate particles with labels in distinct $X_i$'s. 
Thus, 
there is a smallest sufficiently large time $t_\omega$ depending on $\omega$ such that $H_{t_\omega}(\omega)$ is a configuration in $\mc C'$, 
at which point we stop the homotopy $H$. To see that the map \seqm{\mc C}{}{[0,\infty)}, $\omega\mapsto t_\omega$, is continuous, 
notice our homotopy of labels near the basepoint towards it ensures $t_\omega$ varies continuously as we vary the labels by making them 
approach a basepoint. Then given an open neighbourhood $V\subseteq [0,\infty)$ of $t_\omega$, 
for sufficiently small open neighbourhoods $V'\subseteq\mc C$ of $\omega$, the only possible configurations $\aleph$ that could be in $V'$,
and for which $t_\aleph$ could not be in $V$, would have to have at least two particles collided at the origin with labels in the same $X_i$.
But such configurations cannot exist in $\mc C$, since they would not be represented by a point in $W$. 
This defines our first deformation retraction.

Finally, to define a deformation retraction of $\mc C'$ onto $W^\ell$, notice the set of labels of non-degenerate particles 
in a configuration in $\omega\in\mc C'$ corresponds to a distinct point $x$ in $X$ whose coordinates are each either a basepoint or 
one of the labels. So we can think of the set of labels of $\omega$ as a point $x\in X$. 
We define a homotopy $G_t\colon\seqm{\mc C'}{}{\mc C'}$, $t\in[0,\infty)$, 
by again apply the radial expansion $R_t$ to $\omega$ starting at $t=0$, 
meanwhile, homotoping the set of labels towards $W$ when it is 'near' $W$. We must be careful how we perform this homotopy of labels, 
since we must ensure labels of particles that are collided at a point remain represented by a point in $W_n$, 
and that the homotopy does not become discontinuous as a particle enters the boundary and becomes degenerate 
(in which case the coordinate in $x$ corresponding to the label of this particle becomes the basepoint).  
We therefore use a faithful neighbourhood $U$ of $W$ with associated homotopy $h_t\colon\seqm{X}{}{X}$, 
and homotop labels near $W$ towards $W$ by mapping $x$ to $h_{f(t)}(x)$ at each time $t$, where $f(t)=\frac{t}{1+t}$.
This takes care of both of these issues since $h_t$ maps $W$ to $W$ and is defined coordinate-wise as a product of basepoint preserving maps 
$h_{i,t}\colon\seqm{X_i}{}{X_i}$. 
In particular, for each $I\subseteq [n]$, the product $h_{I,t}\colon\seqm{X_I}{}{X_I}$ of the $h_{i,t}$'s over $i\in I$ maps $W_I$ to $W_I$.
So if a subset of $|I|$ non-degenerate particles in $\omega$, each with a label in $X_i$ for each $i\in I$, are represented by a point in $W$,
that is, their set of labels are the coordinates of a distinct point in $W_I$, they remain so under the homotopy $h_{f(t)}$. 
This defines $G_t$.

As those particles unaffected by radial expansion are collided at the origin and already represented by a point in $W$, 
we see that there is a smallest sufficiently large time $t'_{\omega}$ depending on $\omega$
when enough particles are annihilated at the boundary, or have the set of labels homotoped to a point in $W$,
so that the labels of the non-degenerate particles are represented by a point in $W$.
That is, $G_{t'_{\omega}}(\omega)$ becomes a point in the image of $\iota$, at which point we stop the homotopy $G$. 
As before, \seqm{\mc C'}{}{[0,\infty)} mapping $\omega\mapsto t'_{\omega}$ is continuous.  
Here, since $\omega$ is in $\mc C'$, the labels of any subset of particles in $\omega$ are the coordinates of a point 
$p\in X_I$ for some $I\subseteq[n]$, and by Lemma~\ref{LProjection}, the homotopy $h_t$, 
restricting to $h_{I,t}$ on the labels $p$ makes $p$ approach $W_I$ near $W_I$. 
Moreover, since each $h_{i,t}$ by definition homotopes points close to the basepoint in $X_i$ towards it,
any label near a basepoint is homotoped towards it by $h_{f(t)}$.
So $t'_{\omega}$ varies continuously as we homotope a subset of labels of $\omega$ towards a point in $W_I$, 
or each towards a basepoint. 
Then given an open neighbourhood $V\subseteq [0,\infty)$, for sufficiently small open neighbourhoods $V'\subseteq\mc C'$ of $\omega$, 
the only possible configurations $\aleph$ that could be in $V'$, and for which $t'_\aleph$ could not be in $V$, 
would have to have particles collided at the origin that are not represented by a point in $W$,
but these do not exist in $\mc C'$. This defines our second deformation retraction.
\end{proof}

\subsection{The Scanning Map}

Let $\mc C(W)$ denote $\mc C(\mb R,\emptyset;W)$, and $\mc C_i(W)$ the subspace $\mc C_i(\mb R,\emptyset;W)$.  
We define a \emph{scanning map}
$$
\gamma\wcolon\seqmm{\mc C(W)}{\simeq}{\mc C(D^1;W)}{\gamma}{\map(D^1,\bd D^1;W^1)\cong\Omega W^1}
$$
analogously to the one for classical labelled configuration spaces~\cite{MR0331377} as follows.
The first homotopy equivalence in this composite is induced by the map \seqm{\mb R}{}{D^1} sending $t\mapsto \frac{t}{1+\abs{t}}$
(with a choice of homotopy inverse induced by the inclusion \seqm{D^1}{}{\mb R}). 
Let $D_{\varepsilon,z}$ denote the closed interval $[z-\varepsilon,z+\varepsilon]\subseteq\mb R$. 
For any configuration $\omega\in\mc C(W)$, $\gamma(\omega)$ is the map \seqm{D^1}{}{W^1} defined for any $z\in D^1$  
by taking $\gamma(\omega)(z)$ to be the image of $\omega$ under the composite
\begin{equation}
\label{ESection}
\seqmmm{\mc C(W)}{}{\mc C(D_{\varepsilon,z},\bd D_{\varepsilon,z};W)}{\cong}{\mc C(D^1,\bd D^1;W)}{r}{W^1}
\end{equation}
where the first map forgets those particles from $\omega$ that are in $\mb R-D_{\varepsilon,z}$,
the second homeomorphism induced by the homeomorphism $D^1\cong D_{\varepsilon,z}$, and $r$ is the retraction from Lemma~\ref{LBaseCase}.

With regards to labels, $\mc C()$ is a functor from the category of subspaces of 
length $n$ products of basepointed $CW$-complexes, and coordinate-wise maps between them. 
More precisely, given $Y=Y_1\times\cdots\times Y_n$ a product of basepointed $CW$-complexes, 
basepoint preserving maps $f_i\colon\seqm{X_i}{}{Y_i}$, and a subspace $V\subseteq Y$, 
suppose the coordinate-wise map $\bar f=(f_1\times\cdots\times f_n)\colon\seqm{X}{}{Y}$  
restricts to a map $f\colon\seqm{W}{}{V}$. This induces a map
$$
\mc C(f)\wcolon\seqm{\mc C(W)}{}{\mc C(V)}
$$
sending $(z_1,\ldots,z_n;x_1,\ldots,x_n)\mapsto (z_1,\ldots,z_n; f_1(x_1),\ldots,f_n(x_n))$,
and a map 
$$
f^1\wcolon\seqm{W^1}{}{V^1}
$$
sending $((z_1,x_1),\ldots,(z_n,x_n))\mapsto ((z_1,f_1(x_1)),\ldots,(z_n,f_n(x_n)))$.
The scanning map $\gamma$ is then a natural transformation in the homotopy category:

\begin{lemma}
\label{LScanningNatural}
There is a homotopy commutative square
\[\diagram
\mc C(W)\rto^{\mc C(f)}\dto^{\gamma} & \mc C(V)\dto^{\gamma}\\
\Omega W^1 \rto^{\Omega f^1} & \Omega V^1. 
\enddiagram\]
\end{lemma}

\begin{proof}

Since the following square commutes
\[\diagram
W^1\rto^{f^1}\dto^{\iota} & V^1\dto^{\iota}\\
\mc C(D^1,\bd D^1;W)\rto^{\mc C(f)} & \mc C(D^1,\bd D^1;V),
\enddiagram\]
we may replace the homotopy equivalences $\iota$ with their homotopy inverse $r$ to obtain a homotopy commutative square.
Thus we have a diagram
\[\diagram
\mc C(W)\rto^{}\dto^{\mc C(f)} & 
\mc C(D_{\varepsilon,z},\bd D_{\varepsilon,z};W)\rto^-{\cong} &
\mc C(D^1,\bd D^1;W)\rto^-{r}\dto^{\mc C(f)} &
W^1\dto^{f^1}\\
\mc C(V)\rto^{} & 
\mc C(D_{\varepsilon,z},\bd D_{\varepsilon,z};V)\rto^-{\cong} & 
\mc C(D^1,\bd D^1;V)\rto^-{r} &
V^1,
\enddiagram\]
the left square commuting strictly,
and the homotopy $f^1\circ r\simeq r\circ\mc C(f)$ in the right square being independent of $z\in \mb R$.
This then gives a homotopy $\gamma(\mc C(f)(\omega))\simeq \Omega f^1\circ \gamma(\omega)$ of maps \seqm{D^1}{}{W^1} independently of $z$ and $\omega$,
so in turn, it defines a homotopy $\gamma\circ\mc C(f)\simeq \Omega f^1\circ\gamma$.
\end{proof}

\subsection{Suspension Splittings}
\label{SSplitting}

Our goal is to obtain a (single) suspension splitting of $\mc C(W)$. 
Take the quotient space
$$
\mc D_m(W)=\frac{\mc C_m(W)}{\mc C_{m-1}(W)}
$$

\begin{lemma}
\label{LNDR}
The pair $(\mc C_m(W),\mc C_{m-1}(W))$ is an $NDR$-pair. Thus, the sequence 
$$
\seqmm{\mc C_{m-1}(W)}{}{\mc C_m(W)}{}{\mc D_m(W)} 
$$
is a cofibration sequence.   
\end{lemma}

\begin{proof}

Since $X_i$ is a regular $CW$-complex, as in the proof of Lemma~\ref{LBaseCase}, 
we can take a mapping cylinder neighbourhood $M_i$ of the basepoint $\ast\in X_i$ such that
each of the fibers $(0,1]\times\{b\}\subseteq M_i$ lie entirely in a single cell in $X_i$. 
Let $f_i\colon\seqm{X_i}{}{[0,1]}$ be the map given by $f_i(\ast)=0$, $f_i(x)=1$ if $x\nin M_i$, 
and $f_i(x)=s$ if $x$ is at the point $(s,b)$ along a fiber $(0,1]\times\{b\}$. Notice that $f_i$ is continuous as the projection to the parameter $s$.
Define $f\colon\seqm{X_1\vee\cdots\vee X_n}{}{[0,1]}$ by restricting to $f_i$ on $X_i$, 
and let $u\colon\seqm{\mc C_m(W)}{}{[0,1]}$ be given for any $\omega=\{(t_1,x_1),\ldots,(t_m,x_m)\}$ by 
$u(\omega)=0$ if $\omega\in \mc C_{m-1}(W)$, otherwise $u(\omega)=\prod_{i=1}^m f(x_i)$.
Notice $u^{-1}(0)=\mc C_{m-1}(W)$.
Define a homotopy $H\colon \seqm{\mc C_m(W)\times \mb R}{}{\mc C_m(W)}$ by homotoping the labels of a configuration 
$\omega\in \mc C_m(W)$ simultaneously as follows. For each particle $(t,x)\in \omega$, 
we either keep the label $x$ fixed if $x$ is a basepoint or $x$ is not in the interior of any mapping cylinder $M_i$, 
or else we homotop $x$ towards the basepoint along the corresponding fibre of $(0,1]\times\{b\}$ of the mapping cylinder. 
This is done at a speed that tends to zero the closer $x$ is to the boundary of corresponding mapping cylinder.
Since each of the fibers of $M_i$ lies in a single cell of $X_i$, 
and since $W$ as a subcomplex of $X$ is a union of products of cells in the $X_i$'s, 
then under this homotopy of labels collided particles still have their set of labels represented by a point in $W$ as necessary.
We stop homotoping the labels of $\omega$ as soon as there are less than $m$ particles (among our homotoped particles in $\omega$) 
that have non-basepoint labels. 
Thus $H_t$ restricts to the identity on $\mc C_{m-1}(W)$.
Notice at least one of the labels is in the interior of a mapping cylinder $M_i$ when $u(\omega)<1$, 
and $u(\omega)$ approaches $1$ if the labels of $\omega$ all approach points outside these mapping cylinders.
Then for any $(s,b)\in (0,1]\times\{b\}$, there is a sufficiently large time $t$ such that $H_t(u^{-1}([0,s)))$ is contained in $\mc C_{m-1}(W)$.
This therefore defines a neighbourhood deformation retraction for $\mc C_{m-1}(W)$ in $\mc C_m(W)$. 
\end{proof}

Let
$$
V_m(W) = \bigvee_{i=1}^{m}\mc D_i(W).
$$
We now define a map
$$
\zeta\wcolon\seqm{\mc C(W)}{}{C(\mb R;V_\infty(W))},
$$
where the right-hand space is the classical configuration space of non-overlapping particles in $\mb R$ with labels in $V_\infty(W)$.
Since $X_1\vee\cdots\vee X_n$ is a $CW$-complex, we may pick a map 
$$
f\wcolon\seqm{X_1\vee\cdots\vee X_n}{}{[0,1]}
$$ 
such that $f(\ast)=0$ and $f(x)>0$ when $x\neq\ast$.
Let $\zeta_0\colon\seqm{\mc C_0(W)}{}{C(\mb R;V_\infty(W))}$ be given by mapping the 
empty configuration to the empty configuration, and define
$$
\zeta_m\wcolon\seqm{\mc C_m(W)-\mc C_{m-1}(W)}{}{C(\mb R;V_\infty(W))}
$$
for $m>0$ as follows. 
Notice that any configuration $\omega=\{(t_1,x_1),\ldots,(t_m,x_m)\}\in\mc C_m(W)-\mc C_{m-1}(W)$ has only non-degenerate particles,
that is, the label $x_i$ of each particle in $\omega$ is not the basepoint $\ast$.
Since any two non-degenerate particles in a configuration in $\mc C(W)$ cannot collide when their labels are in the same summand $X_l$ 
(so $t_i\neq t_j$ when $x_i,x_j\in X_l-\{\ast\}$ by definition of $\mc C(\mb R,\emptyset;W)$), 
there is a strict total order on $\omega$ given for any $y=(t_i,x_i),y'=(t_j,x_j)\in\omega$ by the relation
\begin{itemize}
\item $y\prec y'$ if and only if either $x_i\in X_{i'}-\{\ast\}$ and $x_j\in X_{j'}-\{\ast\}$ such that $i'<j'$, 
or else both $x_i,x_j\in X_l-\{\ast\}$ for some $l$ and $t_i<t_j$.
\end{itemize}
Given $y\in \omega$, let $\sigma(y)$ be the integer such that $y$ is the $\sigma(y)^{th}$ smallest element in $\omega$ 
with respect to this ordering. Then take the strict total order on the set $2^\omega$ of subsets of $\omega$, given for any $S,S'\subseteq\omega$ by the relation
\begin{itemize}
\item $S\prec S'$ if and only if $\csum{y\in S}{}{2^{\sigma(y)}}<\csum{y'\in S'}{}{2^{\sigma(y')}}$. 
\end{itemize}
This total order does not change as we move these particles and vary their labels in such a way that 
$\omega$ moves along any path in $\mc C_m(W)-\mc C_{m-1}(W)$,
and if we take a subconfiguration $\omega'\subset\omega$ in $\mc C_{m'}(W)-\mc C_{m'-1}(W)$, $m'<m$,
then the ordering on $2^{\omega'}\subset 2^\omega$ is the one inherited from $2^\omega$.
Write $S\preceq S'$ if either $S\prec S'$ or $S=S'$, and let
$$
\eta_{\omega,S'}=\csummulti{S\subseteq\omega}{S\preceq S'}{}{\paren{\cprod{(t_i,x_i)\in S}{}{f(x_i)}}}.
$$
Since each $x_i\neq\ast$, we have $f(x_i)>0$, so $\eta_{\omega,S}<\eta_{\omega,S'}$ whenever $S\prec S'$.
Then we may define
$$
\zeta_m(\omega)=\cset{(\eta_{\omega,S},\chi_S)}{S\subseteq\omega}\quad\in\quad C(\mb R;V_\infty(W))
$$
where the label $\chi_S\in V_\infty(W)$ is the image of the configuration $S\in\mc C_{|S|}(W)$
under the composite of the quotient map and inclusion
\begin{equation}
\label{EQuotientlabel}
\seqmm{\mc C_{|S|}(W)}{}{\mc D_{|S|}(W)}{include}{V_\infty(W)}
\end{equation}
($S\subseteq\omega$ is indeed a configuration in $\mc C_{|S|}(W)$, 
since every subset of a represented multiset of labels is itself represented by a point in $W$).
Notice that when $\omega$ approaches some configuration $\omega'\in\mc C_{m'}(W)-\mc C_{m'-1}(W)$
for some $m'<m$, that is, when some subset $L$ of at least $m-m'$ labels $x_i$ from $\omega$ approaches the basepoint,
then $\prod_{y=(x_i,t_i)\in S}f(x_i)$ approaches $0$ and $\chi_S$ approaches $\ast\in V_\infty(W)$
for all $S\subseteq\omega$ that contain a particle with a label in $L$. 
Therefore, we see that $\zeta_m(\omega)$ approaches $\zeta_{m'}(\omega')$.
Thus, we can define $\zeta$ by restricting to $\zeta_m$ on $\mc C_m(W)-\mc C_{m-1}(W)$, for any $m>0$, 
and $\zeta_0$ on the empty configuration.

Notice $\zeta_m$ maps to the subspace $C(\mb R;V_m(W))\subseteq C(\mb R;V_\infty(W))$.
Restricting $\zeta$ to a subspace $\mc C_i(W)$, there is a commutative diagram 
\begin{equation}
\label{DBeforeAjoint}
\diagram
\mc C_{i-1}(W)\rto^{}\dto^{\zeta} & \mc C_{i}(W)\rto^{}\dto^{\zeta} & 
\mc D_i(W)\dto^{}\\
C(\mb R;V_{i-1}(W))\rto^{} &C(\mb R;V_i(W))\rto^{} & C(\mb R;\mc D_i(W))
\enddiagram
\end{equation}
with the left-hand and right-hand horizontal maps the inclusions and quotient maps respectively,
and the bottom maps induced by the inclusion \seqm{V_{i-1}(W)}{}{V_i(W)} and the quotient map \seqm{V_i(W)}{}{\mc D_i(W)}. 
Since the composite of these two maps is the constant map, 
the composite of the maps in the bottom sequence is the constant map, which gives the right-hand vertical extension. 
Notice this extension sends a (quotiented) configuration $\omega$ to a configuration with a single particle, 
whose label is $\omega$.  
Recall from the theory of classical labelled configuration spaces~\cite{MR0331377} there is a \emph{scanning map} 
$$
\lambda\wcolon\seqm{C(\mb R;V_\infty(W))}{}{\Omega\Sigma V_\infty(W)},
$$
which in fact is a homotopy equivalence. This is also just a special case of our scanning map constructed above.
Here one takes $\Sigma V_\infty(W)$ as the coordinate suspension of $V_\infty(W)$ on a single coordinate,
meaning $\mc C(\mb R;V_\infty(W))=C(\mb R;V_\infty(W))$.
Composing the vertical maps in Diagram~\eqref{DBeforeAjoint} with the corresponding scanning maps $\lambda$, 
and taking the adjoints of these, we obtain a commutative diagram of cofibration sequences
\begin{equation}
\label{DCofibSplitting}
\diagram
\Sigma\mc C_{i-1}(W)\rto^{}\dto^{\zeta'} & \Sigma\mc C_i(W)\rto^{}\dto^{\zeta'} & 
\Sigma\mc D_i(W)\dto^{\simeq}\\
\Sigma V_{i-1}(W)\rto^{} & \Sigma V_i(W)\rto^{} & 
\Sigma\mc D_i(W),
\enddiagram
\end{equation}
the left and middle vertical maps being the restrictions of the adjoint $\zeta'\colon\seqm{\Sigma\mc C(W)}{}{\Sigma V_\infty(W)}$ 
of the composite $\lambda\circ\zeta$, 
and the bottom cofibration sequence the composite of the inclusion and quotient onto the summand $\Sigma(\mc D_i(W))$.
The right-hand vertical extension can be homotoped to the identity using a homotopy to the identity of the composite
\seqmm{\Sigma V_\infty(W)}{}{\mc C(D^1,\bd D^1;V_\infty(W))}{}{\Sigma V_\infty(W)} 
of the inclusion from Lemma~\ref{LBaseCase} and its left homotopy inverse used to define the scanning map.    
Then the middle vertical $\zeta'$ is a homotopy equivalence whenever the left-hand $\zeta'$ is, 
so by induction, we see that \seqm{\Sigma\mc C(W)}{\zeta'}{\Sigma V_\infty(W)} is a homotopy equivalence. 
Thus
\begin{proposition}
\label{PMainSplitting}
There is a splitting
$$
\Sigma\mc C(W)\simeq \cvee{i\geq 1}{}{\Sigma\mc D_i(W)}.
$$~$\qqed$
\end{proposition}

\begin{remark}
If each $X_i=Y$ is a regular $CW$-complex, and $W$ is a polyhedral product of the form $(Y,\ast)^K$,
then by~\cite{2009arXiv0901.2871D}, $\mc C(W)$ is homotopy equivalent to $\Omega (\Sigma Y,\ast)^K$. 
This improves on the stable splitting of $\Omega (\Sigma Y,\ast)^K$ in~\cite{2009arXiv0901.2871D}.
\end{remark}

\subsection{A Refinement}

Suppose $W$ has the simplicial property.
Recall our filtration $\ast=W_0\subseteq W_1\subseteq\cdots\subseteq W_n=W$, where $W_k=\bigcup_{I\subseteq [n],|I|=k}W_I$. 
Again, $\mc C_m(W_k)$ denotes $\mc C_m(\mb R,\emptyset;W_k)$,
these configuration spaces being constructed by thinking of $W_k$ as a subspace of $X$.
We still refer to elements in $\mc D_m(W_k)=\mc C_m(W_k)/\mc C_{m-1}(W_k)$ as configurations of labelled particles, 
in the sense that they are under the image of the quotient map \seqm{\mc C_m(W_k)}{}{\mc D_m(W_k)}.
They must have exactly $m$ particles, 
and are identified with the basepoint when one of them is a degenerate particle, that is, 
when its label is a basepoint. The empty configuration is the basepoint. 
Since $\mc C_m(W_k)$ is a subspace of $SP_m(Y)$, $Y=(\mb R\times X^\vee)/(\mb R\times\{\ast\})$, 
then $\mc D_m(W_k)$ is a subspace of $\hat{SP}_m(Y)=SP_m(Y)/SP_{m-1}(Y)$.

We call a sequence $\mc S=(s_1,\ldots,s_n)$ an $(m,n)$-\emph{partition} of a non-negative integer $m$ 
if $s_1+\cdots+s_n=m$ and each $s_i\geq 0$.
Given such a sequence, let $\mc D_{\mc S}(W_k)$ be the subspace of $\mc D_m(W_k)$ consisting of configurations $\omega$ 
such that either exactly $s_i$ non-degenerate particles in $\omega$ have labels in $X_i$ for each $1\leq i\leq n$, or else $\omega$ is the basepoint. 
Notice $\mc D_m(W_k)$ is the union of $\mc D_{\mc S}(W_k)$ over all $(m,n)$-partitions.
To homotope the label of a particle in a configuration in $\mc D_m(W_k)$ from one summand $X_j$ to another summand $X_i$ we must make at least one label go through the basepoint, in which case the entire configuration must go through the basepoint. 
Therefore $\mc D_{\mc S}(W_k)\cap \mc D_{\mc S'}(W_k)=\{\ast\}$ for any $(m,n)$-partitions $\mc S$ and $\mc S'$, and
\begin{equation}
\label{EPartitions}
\mc D_m(W_k)=\cvee{(m,n)-\mbox{partitions }\mc S}{}{\mc D_{\mc S}(W_k)}
\end{equation}
wedged at the basepoint, and the inclusion 
$$
\iota_{m,k}\wcolon\seqm{\mc D_m(W_{k-1})}{}{\mc D_m(W_k)}
$$ 
induced by \seqm{W_{k-1}}{}{W_k}, and extending $\bar\iota_{m,k}\colon\seqm{\mc C_m(W_{k-1})}{}{\mc C_m(W_k)}$,
restricts to an inclusion 
$$
\iota_{\mc S,k}\colon\seqm{\mc D_{\mc S}(W_{k-1})}{}{\mc D_{\mc S}(W_k)}
$$
for each $(m,n)$-partition $\mc S$. 
Given a basepoint preserving map \seqm{X}{f}{Y} of non-degenerately basepointed spaces, 
let $M(f)=M_f/([0,1]\times\{\ast\})=(Y\cup_f [0,1]\times X)/([0,1]\times\{\ast\})$ be the reduced mapping cylinder of $f$, 
and $\mb C(f)=M(f)/(\{1\}\times X)$ the reduced mapping cone of $f$ (or homotopy cofiber of $f$). 
We write points in $\mb C(f)$ as pairs $(t,y)$, where $t\in [0,1]$, $y\in Y$ if $t=0$, $y\in X$ if $t>0$,
and $(t,y)$ is the basepoint $\ast$ when $t=1$ or $y$ is the basepoint. 
Since $\iota_{m,k}$ is as a wedge sum of $\iota_{\mc S,k}$ over $(m,n)$-partitions $\mc S$,
\begin{equation}
\label{ECofibPartitions}
\mb C(\iota_{m,k}) \cong \cvee{(m,n)-\mbox{partitions }\mc S}{}{\mb C(\iota_{\mc S,k})}.
\end{equation}

Notice that the following property holds when $W$ has the simplicial property. 
\begin{itemize}
\item[$(A)$]
For each $m<k$ a multiset $\{\{x_1,\ldots,x_m\}\}$ is represented by a point in $W_k$ if and only it is represented by a point in $W_m$.
\end{itemize}

\begin{lemma}
\label{LHomeomorphism}
Assume $W$ has the simplicial property. If at most $k-1$ elements in $\mc S$ are nonzero, 
then $\mc D_{\mc S}(W_{k-1})=\mc D_{\mc S}(W_k)$ as subspaces of $\hat{SP}_m(Y)$, 
and $\iota_{\mc S,k}$ is the identity.
In particular, $\mc D_m(W_{k-1})=\mc D_m(W_k)$, $\iota_{m,k}$ is the identity, when $m\leq k-1$. 
\end{lemma}
\begin{proof}
Any non-degenerate configuration $\omega\in\mc D_{\mc S}(W_k)$ has $m$ particles when $\mc S$ is an $(m,n)$-partition,
but since collisions are not allowed between particles with labels in the same $X_i$, 
no more than $k-1$ can collide at a point when at most $k-1$ elements in $\mc S$ are nonzero.
Since the labels of any $j\leq k-1$ particles in $\omega$ that are collided at a point are represented by a point in $W_k$, 
by Property~$(A)$ they are represented by a point in $W_j\subseteq W_{k-1}$. 
Therefore $\omega\in\mc D_{\mc S}(W_{k-1})$. 
\end{proof}

\subsection{Comparing Filtrations}

We again assume each $X_i$ is a connected basepointed regular $CW$-complex,
$W$ is a connected subcomplex of their product $X$ containing the basepoint and having a faithful neighbourhood.
Moreover, we assume $W$ has the simplicial property, and that its subcomplex $W_{n-1}$ also has a faithful neighbourhood.
Note $W^1_k$ is simply connected by Lemma~\ref{LSimplyConnected}. 

Take the $(n,n)$-partition $\mc A=(1,\ldots, 1)$.
Then $\mc D_{\mc A}(W_k)$ is the subspace of $\mc D_n(W_k)$ containing those configurations of $n$ particles 
whose labels are each in a distinct $X_i$. Our goal will be to prove the following.

\begin{proposition}
\label{PCofibDiagram}
There exists a homotopy commutative diagram
\[\diagram
\mb C(\iota_{\mc A,n})\rto^-{\partial} &
\Sigma\mc D_{\mc A}(W_{n-1})\rto^{\Sigma\iota_{\mc A,n}}\dto^{} & \Sigma\mc D_{\mc A}(W_n)\rto^{}\dto^{} & 
\Sigma\mb C(\iota_{\mc A,n})\dto^{\simeq}\\
& \Sigma\mc C_n(W_{n-1})\rto^{\Sigma\bar\iota_{n,n}}\dto^{ev} & \Sigma\mc C_n(W_n)\rto^{}\dto^{ev} & 
\Sigma\mb C(\bar\iota_{n,n})\dto^{\simeq}\\
\Sigma^{n-1}\hat W_n\uuto^{\phi}_{\simeq}\rto^-{} & W^1_{n-1}\rto^{} & W^1_n\rto^{} & \hat W^1_n,
\enddiagram\]
where consecutive maps in each row form homotopy cofibration sequences,
the right-hand vertical homotopy equivalences are extensions induced by 
the homotopy commutativity of the middle squares. 
The top second and third vertical maps are homotopy inverses of the quotient maps coming from the splitting in
Proposition~\ref{PMainSplitting}, and the bottom maps $ev$ are the adjoints 
\seqmm{\Sigma\mc C_n(W_i)}{\Sigma\gamma}{\Sigma\Omega W^1_i}{eval}{W^1_i} 
of the corresponding scanning map $\gamma$. The map $\partial$ is the connecting map,
given by collapsing the subspace $\mc D_{\mc A}(W_n)$ of $\mb C(\iota_{\mc A,n})$. 
\end{proposition}

To some extent, this reduces the study of the filtration of $W^1=W^1_n$ by the spaces $W^1_k$
to the corresponding filtration for $\mc D_{\mc A}(W_n)$.

\begin{corollary}
\label{CInverse}
Considering the following:
\begin{itemize}
\item[(a)] \seqm{W^1_n}{}{\hat W^1_n} has a right homotopy inverse;
\item[(b)] \seqm{\Sigma\mc D_{\mc A}(W_n)}{}{\Sigma\mb C(\iota_{\mc A,n})}
has a right homotopy inverse;
\item[(c)] \seqm{\mb C(\iota_{\mc A,n})}{\partial}{\Sigma\mc D_{\mc A}(W_{n-1})} is nullhomotopic.
\end{itemize}
Then $(c)\Leftrightarrow (b)\Rightarrow (a)$.
\end{corollary}
\begin{proof}
The implication $(b)\Rightarrow(a)$ holds by Proposition~\ref{PCofibDiagram}, 
while $(c)\Rightarrow (b)$ is a standard property of homotopy cofibration sequences.
Suppose $(b)$ holds. Then the homotopy cofibration sequence 
\begin{equation}
\label{EHCofib}
\seqmm{\Sigma\mc D_{\mc A}(W_{n-1})}{\Sigma\iota_{\mc A,n}}{\Sigma\mc D_{\mc A}(W_n)}{}{\Sigma\mb C(\iota_{\mc A,n})}
\end{equation}
trivializes. In particular $\Sigma\iota_{\mc A,n}$ has a left homotopy inverse $\hat\iota$. 
Since \seqmm{\mb C(\iota_{\mc A,n})}{\partial}{\Sigma\mc D_{\mc A}(W_{n-1})}{\Sigma\iota_{\mc A,n}}{\Sigma\mc D_{\mc A}(W_n)}
is a homotopy cofibration sequence, $\Sigma\iota_{\mc A,n}\circ\partial$ is nullhomotopic, 
so $(\hat\iota\circ\Sigma\iota_{\mc A,n})\circ\partial\simeq\partial$ is nullhomotopic. Thus $(b)\Rightarrow (c)$. 

\end{proof}

\begin{corollary}
\label{CcoHspace}
$W^1_n$ is a co-$H$-space if one of the conditions in Corollary~\ref{CInverse} hold for $W_I$ in place of $W$
for each $I\subseteq [n]$.
\end{corollary}

\begin{proof}
Combine Corollary~\ref{CInverse} with Proposition~\ref{PSplitting}.
\end{proof}

\subsection{Proof of Proposition~\ref{PCofibDiagram}}

Take a configuration $\omega=(t_1,\ldots,t_n;x_1,\ldots,x_n)\in \mc D_{\mc A}(W_n)$.
Configurations in $\mc D_{\mc A}(W_n)$ have particles whose labels are each in a distinct $X_i$,
so we can and will assume without loss of generality that $x_i\in X_i$. 
Given $x=(x_1,\ldots,x_n)\in\hat W_n$, then $\{x_1,\ldots,x_n\}$ is represented by $(x_1,\ldots,x_n)\in W_n$,
so there is an inclusion \seqm{\hat W_n}{}{\mc D_{\mc A}(W_n)} given by mapping $x\mapsto (0,\ldots,0;x_1,\ldots,x_n)$.
We can extend this to an inclusion 
$$
\phi\wcolon\seqm{\Sigma^{n-1}\hat W_n}{}{\mb C(\iota_{\mc A,n})}
$$
by mapping the basepoint to the basepoint, otherwise
$$
((t_1,\ldots,t_{n-1}),(x_1,\ldots,x_n))\mapsto \paren{\max\{|t_1|,\ldots,|t_{n-1}|,0\},(t_1,\ldots,t_{n-1},0;x_1,\ldots,x_n)},
$$ 
where each parameter $t_i$ is in the unit $1$-disk $D^1=[-1,1]$. 
Notice $\phi$ is a homeomorphism of $\Sigma^{n-1}\hat W_n$ onto its image, 
so we regard $\Sigma^{n-1}\hat W_n$ as a subspace of $\mb C(\iota_{\mc A,n})$ under this inclusion.

\begin{remark}
Since points in $W_{n-1}$ have at least one coordinate a basepoint,
no more than $n-1$ non-degenerate particles in a configuration in $D_{\mc A}(W_{n-1})$ can collide at a point.
Therefore any non-degenerate point $(t,(t_1,\ldots,t_n;x_1,\ldots,x_n))\in\mb C(\iota_{\mc A,n})$ must have $t=0$ when $t_1=\cdots=t_n$,
i.e., we must be at the base of the mapping cone $D_{\mc A}(W_n)\subseteq\mb C(\iota_{\mc A,n})$.
\end{remark}

\begin{lemma}
\label{LCofib}
There is a deformation retraction of $\mb C(\iota_{\mc A,n})$ onto the subspace $\Sigma^{n-1}\hat W_n$
given by the inclusion $\phi$.
\end{lemma}

\begin{proof}
Let $A$ be the subspace of $\mb C(\iota_{\mc A,n})$ consisting of the points $(t,(t_1,\ldots,t_n;x_1,\ldots,x_n))$
such that $t_n=0$ and $t=\max\{|t_1|,\ldots,|t_{n-1}|,0\}\leq 1$, as well as the basepoint. 
Define a deformation retraction of $\mb C(\iota_{\mc A,n})$ onto $A$ by keeping the basepoint fixed, 
otherwise homotoping a point $(t,(t_1,\ldots,t_n;x_1,\ldots,x_n))\in \mb C(\iota_{\mc A,n})$ to a point in $A$, 
first by shifting the $t_i's$, by linearly homotoping each $t_i$ to $\bar t_i=t_i-t_n$ via $t_{i,s}=(1-s)t_i+s\bar t_i$. 
Then, depending on $t$ and $\beta=\max\{|\bar t_1|,\ldots,|\bar t_{n-1}|,0\}$, we do one of three things:
$(i)$ if $t=\beta$, do nothing since we are done; 
$(ii)$ if $t>\beta$, radially expand the $\bar t_i$'s to $\frac{t}{\beta}\bar t_i$;
$(iii)$ if $t<\beta$, linearly homotope $t$ towards $\beta$, stopping either when $t=\beta$, or when $t=1$ (we are at the basepoint). 
Notice $\beta$ cannot be $0$ in case $(ii)$, otherwise the $\bar t_i's$ are $0$, and by our preceding remark $t=0$. 
The homotopy in case $(iii)$ is possible since $\beta>t\geq 0$, 
so our configuration of particles $(\bar t_i,x_i)$ is necessarily in $D_{\mc A}(W_{n-1})$.
Note no new collisions between particles are created during these homotopies, 
and this homotopy is continuous in the neighbourhood of the basepoint.

Next, we define a deformation retraction of $A$ onto $\Sigma^{n-1}\hat W_n$, homotoping a given point 
$\omega=(t,(t_1,\ldots,t_{n-1},0;x_1,\ldots,x_n))$ in $A$ as follows. 
If $\omega$ is the basepoint, we are done, so assume $\omega$ is not the basepoint. 
Then $t<1$ and each $x_i$ is not the basepoint. 
Since we are in $\mb C(\iota_{\mc A,n})$, we may assume without loss of generality that $x_i\in X_i$, 
so $\hat x=(x_1,\ldots,x_n)$ is a point in $\hat X=X_1\wedge\cdots\wedge X_n$,
and we regard the set of labels $x_i$ as the point $\hat x\in \hat X$. 
Since $\hat x$ is not the basepoint, there is a unique $x\in X$ such that $q(x)=\hat x$,
where \seqm{X}{q}{\hat X} is the quotient map.
To start off the homotopy, if $t=0$, we do nothing.
If $t>0$, we homotope $t=\max\{|t_1|,\ldots,|t_{n-1}|,0\}$ linearly towards $1$ via $t_s=(1-s)t+s$, 
and homotope each $t_i$ for $i\leq n-1$ by radial expansion via $t_{i,s}=\frac{t_s}{t}t_i$
(so that $t_s=\max\{|t_{1,s}|,\ldots,|t_{n-1,s}|,0\}$), while at the same time, 
homotoping the set of labels $\hat x$ by homotoping $x$ towards $W_n$ when it is 'near' $W_n$ to obtain $x_t$ at each time $t$, 
meanwhile letting $\hat x_t=q(x_t)$ to obtain our homotopy $\hat x\mapsto \hat x_t$.
We perform this homotopy of $x$ using a faithful neighbourhood of $W_n$ and applying its associated homotopy 
$h_t\colon\seqm{X}{}{X}$ to $x$, letting $x_t=h_t(x)$, exactly as in the proof of Lemma~\ref{LBaseCase}, 
thereby ensuring that labels of collided particles remain represented by a point in $W_n$ during the homotopy of labels $\hat x$. 
Since $h_t$ is defined coordinate-wise as a product of maps \seqm{X_i}{}{X_i} that are basepoint preserving, 
our homotopy for $\hat x$ is continuous for $\hat x$ in the neighbourhood of the basepoint.
The homotopy stops as soon as $x_t$ becomes a point in $W_n$, in which case $\hat x_t\in \hat W_n$,
or else when $t$ becomes $1$, in which case $\omega$ becomes the basepoint. 
Notice there is no issue of continuity, or the homotopy not stopping, when $t=\max\{|t_1|,\ldots,|t_{n-1}|,0\}$ is at or near $0$. 
This is because all particles are collided when $t=0$.
Thus (as a set) $\hat x$ must be represented by a point in $W_n$, and since there are $n$ labels,
$\hat x$ must be a point $\hat W_n$ and $x$ a point in $W_n$, and the homotopy stops immediately. 
At the same time, points $\omega$ approaching some point with $t=0$ must have their label set $\hat x$ approach a point in $\hat W_n$, 
so $x$ approaches a point in $W_n$, and our homotopy of $x$ near $W_n$ towards $W_n$ ensures continuity. 
We let $t_\omega$ be the smallest time at which this homotopy stops for $\omega$. 
For the same reasons as in the second step of the proof of Lemma~\ref{LBaseCase}, $t_\omega$ depends continuously on $\omega\in A$. 
This then defines our second deformation retraction.

\end{proof}

We will want to describe a right homotopy inverse of $\Sigma\phi$ that it is compatible with the scanning map.
To do so we will need to define a version of the scanning map for quotiented configuration spaces.
Fix $\mc C^\partial = \mc C(D^1,\bd D^1;W_n)$, $\mc C^\partial_i = \mc C_i(D^1,\bd D^1;W_n)$, and   
$\mc D^\partial_i=\mc C^\partial_i/\mc C^\partial_{i-1}$.  
Consider the map
$$
\hat\iota\wcolon\seqm{\hat W^1_n}{}{\mc D^\partial_n}
$$ 
given by sending $((t_1,x_1),\ldots,(t_n,x_n))$ to $(t_1,\ldots,t_n;x_1,\ldots,x_n)$.

Since the incusion \seqm{W^1_n}{\iota}{\mc C^\partial} from Lemma~\ref{LBaseCase} includes into the subspace $\mc C^\partial_n$,
we have an inclusion $\iota_n\colon\seqm{W^1_n}{}{\mc C^\partial_n}$. 
In turn, $\iota_n$ restricts to a map $\iota_{n-1}\colon\seqm{W^1_{n-1}}{}{\mc C^\partial_{n-1}}$. 
Notice then our maps fit into a commutative square 
\begin{equation}
\label{DIota}
\diagram
W^1_n\rto^{q}\dto^{\iota_n} & \hat W^1_n\dto^{\hat\iota}\\
\mc C^\partial_n\rto^{p} & \mc D^\partial_n,
\enddiagram
\end{equation}
where $p$ and $q$ are the quotient maps. 
Like $\iota$, both $\iota_n$ and $\hat\iota$ are homeomorphisms onto their images,
and we think of $W^1_n$ and $\hat W^1_n$ as subspaces of $\mc C^\partial_n$ and $\mc D^\partial_n$ under these inclusions.

\begin{lemma}
\label{LBaseCase2}
There is a deformation retraction of $\mc C^\partial_n$ onto the subspace $W^1_n$ given by the inclusion $\iota_n$.
Moreover, it extends to a deformation retraction of $\mc D^\partial_n$ onto the subspace $\hat W^1_n$ given by the inclusion $\hat\iota$. 
Therefore, $\iota_n$ and $\hat\iota$ are homotopy equivalences.
\end{lemma}

\begin{proof}
Recall that the deformation retraction $H\colon\seqm{\mc C^\partial\times [0,1]}{}{\mc C^\partial}$ 
of $\mc C^\partial$ onto $W^1_n$ in the proof of Lemma~\ref{LBaseCase}. This is defined via a two-step radial expansion.
A configuration $\omega$ is homotoped by radial expanding particles towards the boundary,
meanwhile, in the first step, labels near the basepoint are homotoped towards it until there are no more than $n$ particles left, 
and in the second step, the labels of the remaining particles are homotoped towards $W_n$ when they are near $W_n$, 
this being done via a coordinate-wise homotopy $h\colon\seqm{X\times [0,1]}{}{X}$ that is basepoint preserving on each coordinate.
From this, we see that for any $\omega=(t_1,\ldots,t_m;x_1,\ldots x_m)$ and each $t\in [0,1]$, $H$ satisfies 
$H_t(\omega)=(t'_1,\ldots,t'_m;x'_1,\ldots x'_m)$ such that $t_i\in\bd D^1\Rightarrow t'_i\in\bd D^1$ and $x_i=\ast\Rightarrow x'_i=\ast$.
Degenerate particles thereby remain degenerate during this homotopy.
Thus $H$ restricts to a deformation retraction $H^i$ of $\mc C^\partial_i$ onto $W^1_n$ when $i\geq n$,
and a deformation retraction $\bar H^i$ of $\mc C^\partial_i$ onto $W^1_i$ when $i\leq n-1$. 
Since $H^n\colon\seqm{\mc C^\partial_n\times [0,1]}{}{\mc C^\partial_n}$ restricts to 
$\bar H^{n-1}\colon \colon\seqm{\mc C^\partial_{n-1}\times [0,1]}{}{\mc C^\partial_{n-1}}$, 
there is a well-defined extension $\hat H$ of $H^n$, as in the composite  
$$
\hat H'\colon
\seqmm{\mc D^\partial_n\times [0,1]}{}{(\mc C^\partial_n\times [0,1])/(\mc C^\partial_{n-1}\times [0,1])}{\hat H}{\mc D^\partial_n},
$$ 
with the first map being the quotient map collapsing $\ast\times [0,1]$ onto the basepoint.
Then $\hat H'$ is a deformation retraction of $\mc D^\partial_n$ onto $\hat W^1_n$.


\end{proof}

The quotiented scanning map
$$
\hat\gamma\wcolon\seqm{\mc D_n(W_n)}{}{\Omega\hat W^1_n} 
$$
is defined as follows. Let $\mc D_n(M,N;W_n)=\mc C_n(M,N;W_n)/\mc C_{n-1}(M,N;W_n)$.
Fix any $\varepsilon>0$ (this will be the radius of our scanning interval).
For any configuration $\omega'\in\mc D_n(W_n)$, we define $\hat\gamma(\omega')\colon\seqm{\mb R}{}{\hat W^1_n}$ as the map 
given for any $z\in\mb R$ as the image of $\omega'$ under the composite
$$
\seqmmm{\mc D_n(W_n)}{}{\mc D_n(D_{\varepsilon,z},\bd D_{\varepsilon,z};W_n)}{\cong}{\mc D^\partial_n}{\hat r}{\hat W^1_n}
$$
where $D_{\varepsilon,z}=[z-\varepsilon,z+\varepsilon]\subseteq\mb R$,
and where the first map forgets those particles from $\omega'$ that are in $\mb R-D_{\varepsilon,z}$ 
(thus maps to the basepoint if at least one particle is in $\mb R-D_{\varepsilon,z}$),
the second homeomorphism is induced by the homeomorphism $D_{\varepsilon,z}\cong D^1$, 
and $\hat r$ is the homotopy inverse of $\hat\iota$ taken as the retraction from Lemma~\ref{LBaseCase2}. 
Recall that the regular scanning map $\gamma$ is defined similarly. 
In the case of $\mc C(W_n)$, the restriction of \seqm{\mc C(W_n)}{\gamma}{\Omega W^1_n} to $\mc C_n(W_n)$ 
is defined by sending $\omega\in\mc C_n(W_n)$ to the map $\gamma(\omega)\colon\seqm{\mb R}{}{W^1_n}$,
given for any $z\in\mb R$ as the image of $\omega$ under
$$
\seqmmm{\mc C_n(W_n)}{}{\mc C_n(D_{\varepsilon,z},\bd D_{\varepsilon,z};W_n)}{\cong}{\mc C^\partial_n}{r}{W^1_n}
$$
where the first map forgets particles in $\mb R-D_{\varepsilon,z}$ as before, 
and (since the deformation retraction $H$ in the proof of Lemma~\ref{LBaseCase2} restricts as it does)
$r$ is the homotopy inverse of $\iota_n$ taken as the retraction from Lemma~\ref{LBaseCase2}. 
Replacing the homotopy equivalences $\iota_n$ and $\hat\iota$ in Diagram~\eqref{DIota} with their homotopy inverses $r$ and $\hat r$ 
gives a homotopy commutative diagram, so we obtain a homotopy $q\circ r\simeq \hat r\circ p$. 
Combining this homotopy with the commutative diagram 
\[\diagram
\mc C_n(W_n)\rto^{}\dto^{p'} & \mc C_n(D_{\varepsilon,z},\bd D_{\varepsilon,z};W_n)\rto^-{\cong} & \mc C^\partial_n\dto^{p}\\
\mc D_n(W_n)\rto^{} & \mc D_n(D_{\varepsilon,z},\bd D_{\varepsilon,z};W_n)\rto^-{\cong} & \mc D^\partial_n,
\enddiagram\]
where the vertical maps are the quotient maps, gives a homotopy of $\hat\gamma(p'(\omega))\simeq q(\gamma(\omega))$
that is independent of $\omega$ and $z$. Therefore: 
\begin{lemma}
\label{LScanningNatural2}
There is a homotopy commutative square
\[\diagram
\mc C_n(W_n)\rto^{}\dto^{\gamma} & \mc D_n(W_n)\dto^{\hat\gamma}\\
\Omega W^1_n \rto^{\Omega q} & \Omega\hat W^1_n. 
\enddiagram\]~$\qqed$
\end{lemma}

Let 
$$
\hat ev\colon\seqm{\Sigma\mc D_n(W_n)}{}{\hat W_n^1}
$$ 
be the adjoint of \seqm{\mc D_n(W_n)}{\hat\gamma}{\Omega\hat W_n^1}. 
The composite \seqmm{\Sigma\mc D_n(W_{n-1})}{\Sigma\iota_{n,n}}{\Sigma\mc D_n(W_n)}{\hat ev}{\hat W_n^1} is nullhomotopic as follows.
Since $\hat\gamma$ is given by scanning an $\varepsilon$-neighbourhood (for some fixed $\varepsilon>0$) 
through a configuration $\omega$ of labelled particles in $\mb R$, 
while applying a radial expansion to particles in this neighbourhood to obtain a point along a based loop in $\hat W_n^1$ 
(as defined in the proof of Lemmas~\ref{LBaseCase} and~\ref{LBaseCase2}). 
If no $\varepsilon$-neighbourhood contains all $n$ particles, then every point in the loop is the basepoint in $\hat W_n^1$, 
in which case $\hat{ev}$ will map $(t,\omega)$ to the basepoint in $\hat W_n^1$ for all $t\in[-1,1]$.
Then since no more than $n-1$ particles in a configuration $\omega\in\mc D_n(W_{n-1})$ can collide at a point, 
we can homotope $\hat{ev}\circ\Sigma\iota_{n,n}(t,\omega)$ to the basepoint by homotoping $(t,\omega)$ via radial expansion 
of the $n$ particles in $\omega$ away from $0\in\mb R$,
until no $\varepsilon$-neighbourhood in $\mb R$ contains them all (alternatively, one can scan ever smaller neighbourhoods).
This defines a nullhomotopy of $\hat{ev}\circ\Sigma\iota_{n,n}$, and in turn, this nullhomotopy defines an extension
$$
\rho\wcolon\seqm{\Sigma\mb C(\iota_{n,n})}{}{\hat W^1_n}
$$
of $\hat{ev}$. Take the composite
$$
\rho_{\mc A}\wcolon\seqmm{\Sigma\mb C(\iota_{\mc A,n})}{}{\Sigma \mb C(\iota_{n,n})}{\rho}{\hat W^1_n}
$$
where the first map is the inclusion. This is not the same map as the suspension of the retraction map onto $\hat W^1_n\cong\Sigma^n W_n$ 
in Lemma~\ref{LCofib}, but it will be of use to us since it fits naturally with the maps $ev$ and $\hat{ev}$.

\begin{corollary}
\label{CHomotopic}
The composite $\rho_{\mc A}\circ\Sigma\phi$ is a homotopy equivalence. Thus, $\rho_{\mc A}$ is a homotopy equivalence.
\end{corollary}

\begin{proof}
Take an element $\omega=((t_1,\ldots,t_n),(x_1,\ldots,x_n))\in\hat W^1_n$. 
Since $x=(x_1,\ldots,x_n)$ is in $W_n$, any subset of $\{x_1,\ldots,x_n\}$ is represented the point $x$ in $W_n$.
Then since the homeomorphic image of the inclusion \seqm{\hat W_n}{\hat\iota}{\mc D^\partial_n} in Lemma~\ref{LBaseCase2}
is the subspace of configurations whose particles have labels represented by a point in $W_n$,
and since $\hat\gamma$ is defined in terms of a deformation retraction of $\mc D^\partial_n$ onto this image, 
then the mapping $\omega\mapsto\rho_{\mc A}\circ\Sigma\phi(\omega)$ does not change the labels $x_i$ of $\omega$, that is,
$\rho_{\mc A}\circ\Sigma\phi(\omega)=((\bar t_1,\ldots,\bar t_n),(x_1,\ldots,x_n))$ for some $\bar t_i$. 
Moreover, one can check directly from the construction of $\rho_{\mc A}$ and $\phi(\omega)$ 
that each $\bar t_i$ satisfies the following for each $c=\pm 1$ as $(t_1,\ldots,t_n)$ approaches the boundary of $D^n=[-1,1]^{\times n}$:
$(1)$ $\bar t_n$ approaches $c$ as $t_n$ approaches $c$ (since the scanning maps scans a neighbourhood from right to left in $\mb R$);
and $(2)$ for $i\leq n-1$, if $t_i$ approaches $c$, then either $\bar t_i$ approaches $c$, or else $t_n$ approaches $-c$ (as does $\bar t_n$). 
So if we define a linear homotopy from $\bar t_i$ to $t_i$ via $t_{i,s}=(1-s)\bar t_i+s t_i$, 
then $(1)$ and $(2)$ hold with $t_{i,s}$ in place of $\bar t_i$,  
and we see that $(t_{1,s},\ldots,t_{n,s})$ approaches the boundary of $D^n$ for each $s$ whenever $(t_1,\ldots,t_n)$ does. 
There are therefore no issues of continuity in defining a homotopy from $\rho_{\mc A}\circ\Sigma\phi$ to the identity
by mapping $\omega$ to $((t_{1,s},\ldots,t_{n,s}),(x_1,\ldots,x_n))$ at each time $s$.  
Then using Lemma~\ref{LCofib}, $\rho_{\mc A}$ is a homotopy equivalence since both $\rho_{\mc A}\circ\Sigma\phi$ and $\Sigma\phi$ are.
\end{proof}

\begin{lemma}
\label{LExtend}
Suppose we are given a homotopy commutative diagram of homotopy cofibration sequences 
\[\diagram
Z\rto^{\partial} & X\rto^{f}\dto^{\bar g} & Y\rto^{}\dto^{g} & \Sigma Z\dto^{\simeq}_{\hat g}\\
& A\rto^{h} & B\rto^{} & C,
\enddiagram\]
where the homotopy equivalence $\hat g$ is an induced map of homotopy cofibers, 
and all spaces are homotopy equivalent to simply connected $CW$-complexes.
Then the bottom homotopy cofibration sequence extends to the left to a homotopy cofibration sequence
$$
\seqmm{Z}{\bar g\circ\partial}{A}{h}{B}.
$$ 
\end{lemma}

\begin{proof}
Since the induced map $\hat g$ is a homotopy equivalence, the left-most square $\mc Q$ in the above diagram is a homotopy pushout,
that is, $\mc Q$ is equivalent in the homotopy category to a pushout square whose horizontal maps are cofibrations
(the dual of Proposition $7.6.1$ in~\cite{MR1450595}). 
Likewise, if we take $B'$ to be the homotopy cofibre of $\bar g\circ\partial$,
we obtain a homotopy commutative diagram of homotopy cofibrations 
\[\diagram
Z\rto^{\partial}\ddouble & X\rto^{f}\dto^{\bar g} & Y\rto^{}\dto^{g'} & \Sigma Z\ddouble\\
Z\rto^{\bar g\circ\partial} & A\rto^{h'} & B'\rto^{} & \Sigma Z
\enddiagram\]
where the last two vertical maps are iteratively induced maps of homotopy cofibers, 
and so we see the middle square $\mc Q'$ here is a homotopy pushout as well. 
We obtain a map of homotopy pushouts \seqm{\mc Q}{q}{\mc Q'}, and thus a map \seqm{B'}{p}{B}, 
by mapping $X$, $Y$, and $A$ to themselves via the identity, and since these three are homotopy equivalences, so is $p$. 
From the cube diagram given by $q$ we see that $p\circ h'$ is homotopic to $h$, and we are done.

\end{proof}

\begin{proof}[Proof of Proposition~\ref{PCofibDiagram}]

Since the quotient map $\seqm{\mc C_i(W_n)}{}{\mc D_i(W_n)}$ restricts to the quotient map
$\seqm{\mc C_i(W_k)}{}{\mc D_i(W_k)}$ for $k<n$, by definition of labels in the image of  
$\seqm{\mc C_n(W_n)}{\zeta}{C(\mb R;V_\infty(W_n))}$ in terms of~\eqref{EQuotientlabel} in
Section~\ref{SSplitting}, it follows that $\zeta$ restricts to 
$\seqm{\mc C_n(W_{n-1})}{\zeta}{C(\mb R;V_\infty(W_{n-1}))}$. 
Then composing $\zeta$ and its restriction with the corresponding scanning maps, 
using naturality of the scanning map, and taking the adjoints $\zeta'$ of their composites 
(the homotopy equivalences $\zeta'$ from~\eqref{DCofibSplitting}), we obtain the left commutative 
square in the diagram of homotopy cofibration sequences
\[\diagram
\Sigma\mc C_n(W_{n-1})\rto^{\Sigma\bar\iota_{n,n}}\dto^{\zeta'} & \Sigma\mc C_n(W_n)\rto^{}\dto^{\zeta'} &
\Sigma \mb C(\bar\iota_{n,n})\dto^{\hat\zeta'}\\
\Sigma V_n(W_{n-1})\rto^{\bigvee_i\Sigma\iota_{i,n}} & \Sigma V_n(W_n)\rto^{} & \cvee{1\leq i\leq n}{}{\Sigma \mb C(\iota_{i,n})}. 
\enddiagram\]
Here, the bottom sequence is the wedge sum of homotopy cofibration sequences 
\seqmm{\Sigma\mc D_i(W_{n-1})}{\Sigma\iota_{i,n}}{\Sigma\mc D_i(W_n)}{}{\Sigma\mb C(\iota_{i,n})} 
for $1\leq i\leq n$, thus giving the right vertical extension $\hat\zeta'$.
Since the left and middle vertical maps are homotopy equivalences, so is $\hat\zeta'$.
But since $\iota_{m,k}$ is a homeomorphism when $m<k$ by Lemma~\ref{LHomeomorphism}, the only non-contractible homotopy cofibre
$\Sigma \mb C(\iota_{i,n})$ between $1\leq i\leq n$ is $\Sigma \mb C(\iota_{n,n})$,
implying that the composite 
$$
\hat\nu\wcolon\seqmm{\Sigma \mb C(\iota_{n,n})}{include}{\cvee{1\leq i\leq n}{}
{\Sigma \mb C(\iota_{i,n})}}{(\hat\zeta')^{-1}}{\Sigma \mb C(\bar\iota_{n,n})}
$$ 
is a homotopy equivalence, where $(\hat\zeta')^{-1}$ is a homotopy inverse of $\hat\zeta'$. 
Replacing the vertical homotopy equivalences $\zeta'$, $\hat\zeta'$ with their inverses, 
and restricting to the homotopy cofibration sequence   
\seqmm{\Sigma\mc D_n(W_{n-1})}{\Sigma\iota_{n,n}}{\Sigma\mc D_n(W_n)}{}{\Sigma \mb C(\iota_{n,n})},
gives a homotopy commutative diagram of homotopy cofibrations
\begin{equation}
\label{DTop}
\diagram
\Sigma\mc D_n(W_{n-1})\rto^{\Sigma\iota_{n,n}}\dto^{\nu} & \Sigma\mc D_n(W_n)\rto^{}\dto^{\nu} & 
\Sigma \mb C(\iota_{n,n})\dto^{\hat\nu}_\simeq\\
\Sigma\mc C_n(W_{n-1})\rto^{\Sigma\bar\iota_{n,n}} & \Sigma\mc C_n(W_n)\rto^{} & \Sigma \mb C(\bar\iota_{n,n}).
\enddiagram
\end{equation}
In turn, the top homotopy cofibration sequence splits as the wedge sum of homotopy cofibrations sequences
$$
\seqmm{\Sigma\mc D_{\mc S}(W_{n-1})}{\Sigma\iota_{\mc S,n}}{\Sigma\mc D_{\mc S}(W_n)}{}{\Sigma \mb C(\iota_{\mc S,n})}
$$
over all $(n,n$)-partitions $\mc S$,
and since $\iota_{\mc S,k}$ is a homeomorphism when at most $k-1$ elements in an $(m,n)$-partition $\mc S$ are non-zero,
the only non-contractible homotopy cofibre in this splitting is $\Sigma \mb C(\iota_{\mc A,n})$,
so the inclusion 
$$
\seqm{\Sigma \mb C(\iota_{\mc A,n})}{}{\Sigma \mb C(\iota_{n,n})}
$$ 
is a homotopy equivalence. Restricting to the sequence $\mc S=\mc A$, 
we obtain the top homotopy commutative diagram of homotopy cofibrations in the statement of the proposition.  

As for the bottom diagram of homotopy cofibrations, 
note by adjointing the square in Lemma~\ref{LScanningNatural2}, we have a commutative diagram
\[\diagram
\Sigma\mc C_n(W_n)\rto^{}\dto^{ev} & 
\Sigma\mc D_n(W_n)\rto^{}\dto^{\hat{ev}} &
\Sigma \mb C(\iota_{n,n})\dlto^{\rho}\\
W^1_n\rto^{} & \hat W^1_n.
\enddiagram\]
The top left horizontal quotient map fits into a commutative diagram of homotopy cofibrations
\[\diagram
\Sigma\mc C_n(W_{n-1})\rto^{\Sigma\bar\iota_{n,n}}\dto^{} & \Sigma\mc C_n(W_n)\rto^{}\dto^{} & 
\Sigma\mb C(\bar\iota_{n,n})\dto^{\hat q}\\
\Sigma\mc D_n(W_{n-1})\rto^{\Sigma\iota_{n,n}} & \Sigma\mc D_n(W_n)\rto^{} & 
\Sigma\mb C(\iota_{n,n})
\enddiagram\]
induced by the left square, where the vertical maps are the quotient maps. 
Combining the right square with the first diagram, we obtain a commutative diagram of homotopy cofibration sequences
\begin{equation}
\label{DBottom}
\diagram
\Sigma\mc C_n(W_{n-1})\rto^{\Sigma\bar\iota_{n,n}}\dto^{ev} & 
\Sigma\mc C_n(W_n)\rto^{}\dto^{ev} & 
\Sigma\mb C(\bar\iota_{n,n})\dto^{\bar\rho}\\
W^1_{n-1}\rto^{} & W^1_n\rto^{} & \hat W^1_n,
\enddiagram
\end{equation}
where $\bar\rho$ is the composite 
$$
\bar\rho\wcolon\seqmm{\Sigma \mb C(\bar\iota_{n,n})}{\hat q}{\Sigma \mb C(\iota_{n,n})}{\rho}{\hat W^1_n}.
$$
To see that $\bar\rho$ is a homotopy equivalence, note by Corollary~\ref{CHomotopic} the composite
$$
\rho_{\mc A}\colon\seqmm{\Sigma \mb C(\iota_{\mc A,n})}{\simeq}{\Sigma \mb C(\iota_{n,n})}{\rho}{\hat W^1_n}
$$
is a homotopy equivalence. Since we saw above that the first inclusion is a homotopy equivalence, 
then so is $\rho$.
Now the composite of \seqm{\Sigma\mc C_n(W_n)}{\zeta'}{\Sigma V_n(W_n)} 
with the quotient map onto the summand $\Sigma\mc D_n(W_n)$ is the left homotopy inverse 
of the map $\nu$ in diagram~(\ref{DTop}), 
and by the right-hand commutative square in diagram~(\ref{DCofibSplitting}) in Section~\ref{SSplitting}, 
this composite is equal to the composite 
$$
\seqmm{\Sigma\mc C_n(W_n)}{q}{\Sigma\mc D_n(W_n)}{\simeq}{\Sigma\mc D_n(W_n)}
$$
of $q$ and some self homotopy equivalence. 
Thus, $q$ is also a left homotopy inverse of $\nu$, and similarly for $W_{n-1}$ in place of $W_n$. 
Then composing (\ref{DTop}) with the commutative diagram
\[\diagram
\Sigma\mc C_n(W_{n-1})\rto^{\Sigma\bar\iota_{n,n}}\dto^{q} & \Sigma\mc C_n(W_n)\rto^{}\dto^{q} & 
\Sigma \mb C(\bar\iota_{n,n})\dto^{\hat q}\\
\Sigma\mc D_n(W_{n-1})\rto^{\Sigma\iota_{n,n}} & \Sigma\mc D_n(W_n)\rto^{} & \Sigma \mb C(\iota_{n,n})
\enddiagram\]
gives a commutative diagram of homotopy cofibrations
\[\diagram
\Sigma\mc D_n(W_{n-1})\rto^{\Sigma\bar\iota_{n,n}}\dto^{q\circ\nu} & \Sigma\mc D_n(W_n)\rto^{}\dto^{q\circ\nu} & 
\Sigma \mb C(\iota_{n,n})\dto^{\hat q\circ\hat\nu}\\
\Sigma\mc D_n(W_{n-1})\rto^{\Sigma\iota_{n,n}} & \Sigma\mc D_n(W_n)\rto^{} & \Sigma \mb C(\iota_{n,n}).
\enddiagram\]
Since both composites $q\circ\nu$ are homotopic to the identity, 
$\hat q\circ\hat\nu$ is a homotopy equivalence. But $\hat\nu$ is a homotopy equivalence, 
so $\hat q$ is homotopy equivalence, and therefore $\bar\rho=\rho\circ\hat q$ is a homotopy equivalence since $\rho$ is.
Thus, we have obtain the right-side of the bottom homotopy commutative diagram of homotopy cofibrations in the 
statement of the proposition. 

It remains to show that the bottom cofibration sequence in the proposition extends to the left, 
and to obtain the left-hand square involving $\phi$. This is clear when $n=1$, since $W^1_{n-1}=\ast$ in this case. 
Since $\Sigma^{n-1}\hat W_n$ is simply connected when $n\geq 2$, and $\phi$ is a homotopy equivalence by Lemma~\ref{LCofib}, 
the $n\geq 2$ case follows by Lemma~\ref{LExtend}.
\end{proof}

\section{The Golod Property}
\label{SGolod}

From now on we fix $W=(D^1,S^0)^K$ as a subspace of $X=(D^1)^{\times n}$, where $K$ on vertex set $[n]$
has no ghost vertices. Thus, $W^1=(D^2,S^1)^K$. Since $W$ is connected, has the simplicial property, 
and each $W_k$ has a faithful neighbourhood by Examples~\ref{EFaithful} and~\ref{EFaithful2}, 
all of our previous results apply.

The following confirms a large primes version of Conjecture~\ref{CGolodConj},  
generalising Berglund's rational result in~\cite{BerglundRational}.

\begin{theorem}
\label{TMainStatementPrimes}
Localised at any sufficiently large prime $p$, 
$(D^2,S^1)^K$ is homotopy equivalent to a wedge of spheres if and only if $K$ is Golod over $\mb Z_p$.
\end{theorem} 
\begin{proof}
The right-hand implication is clear. 
For the left-hand, recall that a space $Y$ is a co-$H$-space if and only if the evaluation map \seqm{\Sigma\Omega Y}{ev}{Y} 
has a right homotopy inverse~\cite{MR0267582}. A suspension $\Sigma Y$ of a simply connected finite $CW$-complex $Y$ is homotopy equivalent to a wedge of spheres when localised at sufficiently large primes (see for example~\cite{MR1802847, Anick3}). 

The moment-angle complex $W^1$ is a $2$-connected finite $CW$-complex. 
Then $\Omega W^1$ has finite type and the skeleta of $\Omega W^1$ are finite $CW$-complexes,
and suspending, the skeleta of $\Sigma\Omega W^1$ are homotopy equivalent to wedges of spheres when localised at sufficiently large primes. 
Also, the evaluation map \seqm{\Sigma\Omega W^1}{ev}{W^1} induces a surjection on $\mb Z_p$-homology whenever 
$H^+(W^1;\mb Z_p)$ has trivial cup products and Massey products, i.e. when $K$ is Golod over $\mb Z_p$
(see Theorem~$3.2$ and the proof of Theorem~$3.1$ in~\cite{arXiv:1211.0873}).
Then $H_*(W^1;\mb Z_p)$ being a finite dimensional vector space, for sufficiently large $k$ and prime $p$, 
we can construct a right homotopy inverse for $ev$ localised at $p$ by taking the inclusion of a skeleton $\seqm{(\Sigma\Omega W^1)^{(k)}}{}{\Sigma\Omega W^1}$
and selecting the appropriate spheres from a $p$-local splitting of $(\Sigma\Omega W^1)^{(k)}$.
In particular, $W^1$ is homotopy equivalent to a wedge of spheres localised at $p$. 
\end{proof}

We now focus on the integral case. 
An \emph{ordered partition} of a finite set $I$ is a sequence of disjoint non-empty subsets $\mc I=(I_1,\ldots,I_m)$ of $I$ such that 
$I=I_1\cup\cdots\cup I_m$.
If $I$ is a subset of the integers $[n]$ with $k=|I|$ elements, and $\mc S=(s_1,\ldots,s_k)$ is any sequence of real numbers, 
let $i_\ell$ denote the $\ell^{th}$ smallest element in $I$, $s'_j$ denote the $j^{th}$ smallest element in the set $S=\{s_1,\ldots,s_k\}$, 
and $m=|S|$ be the number of distinct elements in $\mc S$.
Then assign to $\mc S$ the ordered partition $I_{\mc S}=(I_1,\ldots,I_m)$ of $I$ where $I_j=\cset{i_\ell\in I}{s_\ell=s'_j}$. 
When $I=[n]$ we have $k=n$ and $I_j=\cset{\ell}{s_\ell=s'_j}$,
so for example, if $I=[4]$ and $\mc S=(-1,\pi,-1,0)$, then $[4]_{\mc S}=(\{1,3\},\{4\},\{2\})$.

We think of the topological $(n-1)$-simplex $|\Delta^{n-1}|$ and its subcomplexes $|K|$ as being without basepoint.
The first suspensions $\Sigma|\Delta^{n-1}|$ and $\Sigma|K|$ are therefore unreduced. 
Again, the suspensions $\Sigma|\Delta^{n-1}|$ and $\Sigma|K|$ are themselves basepointed, 
with the basepoint $\ast_{-1}$ the tip of the double cone corresponding to the basepoint $-1\in D^1=[-1,1]$.
Higher suspensions are therefore reduced. This means that any point $(t_1,\ldots,t_n,z)\in \Sigma^n|K|$
is identified with the basepoint $\ast$ if and only if $t_n=-1$ or $t_i=\pm 1$ for some $i<n$.

Take the diagonal
$$
\vartriangle_n=\cset{(x_1,\ldots,x_n)\in\mb R^n}{x_1=\cdots=x_n}
$$
and consider the following subspace of the smash product $\mc P_n=(\mb R^n-\vartriangle_n)\wedge\Sigma|\Delta^{n-1}|$:
$$
\mc Q_K = 
\cunionmulti{y\in(\mb R^n-\vartriangle_n)}{(I_1,\ldots,I_m)=[n]_{y}}{}{\{y\}\wedge\Sigma |K_{I_1}\ast\cdots\ast K_{I_m}|},
$$
where $[n]_y$ is a partition $[n]_{(y_1,\ldots, y_n)}$ of $[n]$ as defined above.
Here we took $\mb R^n-\vartriangle_n$ to be without a basepoint, 
so $\mc P_n$ is the half-smash product 
$((\mb R^n-\vartriangle_n)\times\Sigma|\Delta^{n-1}|)/(\mb R^n-\vartriangle_n)\times\{\ast\}$,
and is basepointed.

\begin{definition}
A simplicial complex $K$ on vertex set $[n]$ is \emph{weakly coherently homotopy Golod} if $K$ is a single vertex, 
or (recursively) $K\backslash\{i\}$ is weakly coherently homotopy Golod for each $i\in [n]$, 
and the map
$$
\Phi_K\wcolon\seqm{\Sigma^n |K|}{}{\Sigma\mc Q_K}
$$
given for any $z\in |K|$, $t_1,\ldots t_{n-1},t\in[-1,1]$, and $\beta=\max\{|t_1|,\ldots,|t_{n-1}|,0\}$ by 
$$
\Phi_K(t_1,\ldots,t_{n-1},t,z)=
\paren{2\beta-1,(t_1,\ldots,t_{n-1},0),(t,z)}  
$$
is nullhomotopic.
\end{definition}

\begin{theorem}
\label{TMainStatement1}
If $K$ is weakly coherently homotopy Golod, then $(D^2,S^1)^K$ is a co-$H$-space.
\end{theorem} 

We will need a bit of setup before we can prove this.
Recall that $\hat W$ is the image of $W$ under the quotient map $(D^1)^{\times n}\longrightarrow (D^1)^{\wedge n}$. 
To prove this theorem we will need to use a simplicial description of $\hat W$.
In~\cite{MR2673742} the theory of \emph{diagrams of spaces} was used to show the existence of a homeomorphism $\Sigma |K|\cong\hat W$,
but for our purposes, we give an explicit description of this homeomorphism as follows. 
Take the simplex $\Delta^{n-1}$ to be on vertex set $[n]$, and think of its geometric realisation as the subspace
$$
|\Delta^{n-1}|=\cset{(t_1,\ldots,t_n)\in\mb R^n}{0\leq t_i\leq 1\mbox{ and }t_1+\cdots+t_n=1}
$$
with each face $\sigma\in\Delta^{n-1}$ corresponding to the subspace
$|\sigma|=\cset{(t_1,\ldots,t_n)\in |\Delta^{n-1}|}{t_i=0\mbox{ if }i\nin\sigma}$.
Take the subspaces of $|\Delta^{n-1}|$ and $(D^1)^{\wedge n}$:
$$
U_{i}=\cset{(t_1,\ldots,t_n)\in |\Delta^{n-1}|}{t_i=\max\{t_1,\ldots,t_n\}}, 
$$
$$
V_{i,t}=\cset{(t'_1,\ldots,t'_n)\in (D^1)^{\wedge n}}{t'_i=t=\min\{t'_1,\ldots,t'_n\}}. 
$$
Consider a homeomorphism 
\begin{equation}
\label{EHomeomorphism}
h\wcolon\seqm{\Sigma|\Delta^{n-1}|}{\cong}{(D^1)^{\wedge n}}  
\end{equation}
given as follows. For $-1<t<1$, take the homeomorphism $h_{i,t}\colon\seqm{U_i}{}{V_{i,t}}$
given by mapping 
$$
(t_1,\ldots,t_n)\mapsto \paren{t+(1-t)\frac{t_i-t_1}{t_i},\ldots,t+(1-t)\frac{t_i-t_n}{t_i}},
$$
and its inverse homeomorphism $h^{-1}_{i,t}\colon\seqm{V_{i,t}}{}{U_i}$ is given by mapping  
$$
(t'_1,\ldots,t'_n)\mapsto \paren{\frac{1-t'_1}{n-\Sigma_{k=1}^n{t'_k}},\ldots,\frac{1-t'_n}{n-\Sigma_{k=1}^n{t'_k}}}
$$
(to see that $h_{i,t}\circ h^{-1}_{i,t}$ and $h^{-1}_{i,t}\circ h_{i,t}$ are the identity,
use the fact that here $t'_i=t$ and $1=t_1+\cdots+t_n$).
Notice that $h_{i,t}$ agrees with $h_{j,t}$ on $U_i\cap U_j$, and $h^{-1}_{i,t}$ agrees with $h^{-1}_{j,t}$ on $V_{i,t}\cap V_{j,t}$. 
Thus, for each $t\in(-1,1)$, we can piece together the homeomorphisms $h_{i,t}$ over all $i\in[n]$ to give a homeomorphism
$$
h_t\wcolon\seqm{|\Delta^{n-1}|=\paren{\cunion{i=1}{n}{U_i}}}{\cong}{\paren{\cunion{i=1}{n}{V_{i,t}}}=V_t}
$$   
defined unambiguously by mapping $z\mapsto h_{i,t}(z)$ if $z\in U_i$,
with inverse $h^{-1}_t$ defined by mapping $z'\mapsto h^{-1}_{i,t}(z')$ if $z'\in V_{i,t}$.
Notice that $(D^1)^{\wedge n}$ is the disjoint union of $V_t$ over all $t\in[-1,1]$, 
with $V_1$ being a single point $(1,\ldots,1)$, and $V_{-1}$ being the basepoint,
and the maps $h_t$ vary continuously with respect to $t$. 
Thus, we define the homeomorphism $h$ for any $(t,z)\in\Sigma|\Delta^{n-1}|$ by
$$
h(t,z)= 
\begin{cases}
h_t(z) & \mbox{if }-1<t<1;\\
(1,\ldots,1) & \mbox{if }t=1;\\
\ast & \mbox{if }t=-1.
\end{cases}
$$
The crucial property of $h$ is that it restricts on an $m$-dimensional face $\sigma$ of $\Delta^{n-1}$ to a homeomorphism
$$
h_\sigma\wcolon\seqm{\Sigma|\Delta^m|\cong \Sigma|\sigma|}{\cong}{(Y^\sigma_1\wedge\cdots\wedge Y^\sigma_n)\cong (D^1)^{\wedge (m+1)}},  
$$
where 
$$
Y^\sigma_i=
\begin{cases}
D^1 & \mbox{if }i\in\sigma\\
\{1\} & \mbox{if }i\nin\sigma.
\end{cases}
$$
Then piecing these together, for any subcomplex $L$ of $\Delta^{n-1}$, $h$ restricts to a homeomorphism
$$
\seqm{\Sigma|L|=\cunion{\sigma\in L}{}{\Sigma|\sigma|}}{\cong}{\cunion{\sigma\in L}{}{Y^\sigma_1\wedge\cdots\wedge Y^\sigma_n}}.  
$$
In particular, when $W=(D^1,S^0)^K$,  
$$
\hat W=\bigcup_{\sigma\in K}Y^\sigma_1\wedge\cdots\wedge Y^\sigma_n,
$$
so $h$ restricts to a homeomorphism $\Sigma |K|\cong\hat W$.
More generally, for any partition $(I_1,\ldots,I_m)$ of $[n]$, 
we have an inclusion $\hat W_{I_1}\wedge\cdots\wedge\hat W_{I_m}$ into $(D^1)^{\wedge n}$ given by mapping
elements by permuting coordinates, so that if a $j^{th}$ coordinate is the coordinate in a factor $\hat W_{I_k}$ that corresponds to $i\in I_k$,
upon mapping it becomes the $i^{th}$ coordinate. This then gives a homeomorphism onto its image
$$
\hat W_{I_1}\wedge\cdots\wedge \hat W_{I_m}\cong \cunion{\sigma_1\in K_{I_1},\cdots,\sigma_m\in K_{I_m}}{}
{Y^{(\sigma_1\cup\cdots\cup\sigma_m)}_1\wedge\cdots\wedge Y^{(\sigma_1\cup\cdots\cup\sigma_m)}_n},
$$
so more generally $h$ restricts to a homeomorphism
\begin{equation}
\label{EGeneralHomeo}
\seqm{\Sigma |K_{I_1}\ast\cdots\ast K_{I_m}|=\cunion{\sigma_1\in K_{I_1},\cdots,\sigma_m\in K_{I_m}}{}
{\Sigma|\sigma_1\cup\cdots\cup\sigma_m|}}{\cong}{\hat W_{I_1}\wedge\cdots\wedge \hat W_{I_m}}.
\end{equation}
In this manner, we will think of $\hat W_{I_1}\wedge\cdots\wedge\hat W_{I_m}$ as a subspace of $(D^1)^{\wedge n}$.

\begin{proof}[Proof of Theorem~\ref{TMainStatement1}]
Let $W=(D^1,S^0)^K$. At most $n-1$ particles collide in a configuration in $\mc D_{\mc A}(W_{n-1})$, 
and if we look at the subspace of all configurations of $n$ particles fixed at some $t_1,\ldots,t_n\in\mb R$
(that is, all possible labels for particles fixed in this way), 
then this subspace is homeomorphic to $\{y\}\wedge\hat W_{I_1}\wedge\cdots\wedge \hat W_{I_m}$, 
where $y=(t_1,\ldots,t_n)$, and $(I_1,\ldots,I_m)=[n]_{y}$. 
This in turn is homeomorphic to $\{y\}\wedge\Sigma |K_{I_1}\ast\cdots\ast K_{I_m}|$ 
via the restriction of the homeomorphism \seqm{\Sigma|\Delta^{n-1}|}{h}{(D^1)^{\wedge n}} in~(\ref{EHomeomorphism}).
Thus, we see that the map 
$$
\tilde h\wcolon\seqm{\mc Q_K}{}{\mc D_{\mc A}(W_{n-1})} 
$$
sending $((t_1,\ldots,t_n),x)\mapsto(t_1,\ldots,t_n;x_1,\ldots,x_n)$, where $(x_1,\ldots,x_n)=h(x)$, is a homeomorphism.
Notice this map fits into a commutative square
\[\diagram
\Sigma^n |K|\rrto^{\Phi_K}\dto^{\Sigma^{n-1} h}_{\cong} && \Sigma\mc Q_K\dto^{\Sigma\tilde h}_{\cong}\\
\Sigma^{n-1}\hat W\rto^{\phi}_{\simeq} & \mb C(\iota_{\mc A,n})\rto^-{\partial} & \Sigma\mc D_{\mc A}(W_{n-1}),
\enddiagram\]
where $\partial$ is the connecting map given by collapsing the subspace $\mc D_{\mc A}(W_n)$,
and $\phi$ the homotopy equivalence from Lemma~\ref{LCofib}.

Induct on number of vertices. Assume the statement holds for simplicial complexes on less than $n$ vertices.
Suppose $K$ is weakly coherently homotopy Golod. 
Directly from definition, each $K_I$ is weakly coherently homotopy Golod, 
so by our inductive assumption $W^1_I$ is a co-$H$-space for $|I|<n$. Then using Proposition~\ref{PSplitting}, 
the quotient map \seqm{W^1_I}{}{\hat W^1_I} has a right homotopy inverse for $|I|<n$, 
and to show $W^1$ is a co-$H$-space, it remains to show that \seqm{W^1}{}{\hat W^1} has a right homotopy inverse.
Since $\Phi_K$ is nullhomotopic and $\phi$ is a homotopy equivalence, 
$\partial$ is nullhomotopic by the above commutative square, 
and so \seqm{W^1}{}{\hat W^1} has a right homotopy inverse by Corollary~\ref{CInverse}.
\end{proof}

\begin{example}
Let $K$ be two disjoint vertices. 
In this case the right-hand factor in the definition of $\mc Q_K$ is in all cases homeomorphic to $D^2$, 
while left-hand factor is $\mb R^2$ minus the diagonal. Then $\mc Q_K$ is homotopy equivalent to the left half smash product $(\ast\sqcup\ast)\ltimes D^2$, 
and as such, it is contractible, and $\Phi_K$ is nullhomotopic.  
Therefore $K$ is weakly coherently homotopy Golod and $(D^2, S^1)^K$ is a co-$H$-space - as expected since it is homeomorphic to a $3$-sphere.
\end{example}

The weakly coherent homotopy Golod condition - as it stands - does not look near enough to 
the Golod condition for our liking.
With this in mind, we start by considering complexes $K$ for which each inclusion
\seqm{|K_{I\cup J}|}{\iota_{I,J}}{|K_I\ast K_J|} is nullhomotopic after appropriate suspension.
This is just a homotopy version of the Golod property without consideration for Massey products, 
so appropriately, we call it the \emph{homotopy Golod} condition.
It turns out it is a necessary condition for $(D^2,S^1)^K$ to be a co-$H$-space
(and so it might be a good avenue for finding a counter-example to Conjecture~\ref{CGolodConj}).

\begin{proposition}
\label{PNecessary}
If $(D^2,S^1)^K$ is a co-$H$-space, then each one of the suspended inclusions 
$$
\Sigma^{|I\cup J|+1}\iota_{I,J}\wcolon\seqm{\Sigma^{|I\cup J|+1} |K_{I\cup J}|}{}{\Sigma^{|I\cup J|+1} |K_I\ast K_J|}
$$
is nullhomotopic for every disjoint non-empty $I,J\subsetneq [n]$.
\end{proposition}

\begin{proof}

Recall that if $Y$ is a co-$H$-space with comultiplication \seqm{Y}{\psi}{Y\vee Y} if and only if 
the diagonal map \seqm{Y}{\vartriangle}{Y\times Y} is homotopic to the composite
\seqmm{Y}{\psi}{Y\vee Y}{include}{Y\times Y}. In particular, this implies the reduced diagonal map
$$
\bar\vartriangle\wcolon\seqmmm{Y}{\vartriangle}{Y\times Y}{}{(Y\times Y)/(Y\vee Y)}{\cong}{Y\wedge Y}
$$
is nullhomotopic whenever $Y$ is a co-$H$-space.
As before, let $W=W_n=(D^1,S^0)^K$, in which case $W^1=W^1_n=(D^2,S^1)^K$, and suppose $W^1$ is a co-$H$-space. 
By Proposition~\ref{PSplitting}, $W^1_{I\cup J}$ is a co-$H$-space for each disjoint non-empty $I,J\subsetneq [n]$, 
and so \seqm{W^1_{I\cup J}}{\bar\vartriangle}{W^1_{I\cup J}\wedge W^1_{I\cup J}} is nullhomotopic. 
Then the composite 
$$
\seqmmm{W^1_{I\cup J}}{\bar\vartriangle}{W^1_{I\cup J}\wedge W^1_{I\cup J}}{}{W^1_I\wedge W^1_J}{}{\hat W^1_I\wedge\hat W^1_J}
$$
is nullhomotopic (where the second last map is the smash of the coordinate-wise projection maps onto $W^1_I$ and $W^1_J$, 
and the last map is the smash of quotient maps) and moreover, is equal to the composite
$$
\seqmm{W^1_{I\cup J}}{quotient}{\hat W^1_{I\cup J}}{\iota^1}{\hat W^1_I\wedge\hat W^1_J}
$$
where the last map $\iota^1$ is the coordinate-wise inclusion 
(given by mapping $(x_1,\ldots,x_n)$ to the pair $((y_1,\ldots,y_{|I|}),(z_1,\ldots,z_{|J|}))$ 
such that $y_i=x_k$ if $k\in I$ is the $i^{th}$ largest element, 
or else $z_j=x_k$ if $k\in J$ is the $j^{th}$ largest element), so this composite is nullhomotopic as well. 
But since $W^1_{I\cup J}$ is a co-$H$-space, by Proposition~\ref{PSplitting} the quotient map \seqm{W^1_{I\cup J}}{}{\hat W^1_{I\cup J}} 
has a right homotopy inverse, so $\iota^1$ must be nullhomotopic. 
Notice that with respect to the homeomorphisms $\hat W^1_{I'}\cong\Sigma^{|I'|}\hat W_{I'}$ given in Section~\ref{SSplitting}, 
$\iota^1$ is the $(|I\cup J|)$-fold suspension of the coordinate-wise inclusion $\iota^0\wcolon\seqm{\hat W_{I\cup J}}{}{\hat W_I\wedge\hat W_J}$,
and with respect to the homeomorphisms~\eqref{EGeneralHomeo}, $\iota^0$ is just the suspended inclusion 
\seqm{\Sigma|K_{I\cup J}|}{\Sigma\iota_{I,J}}{\Sigma |K_I\ast K_J|}. Thus, we see that $\Sigma^{|I\cup J|+1}\iota_{I,J}$ is nullhomotopic.

\end{proof}

In a follow-up paper~\cite{BebenGrbic2}, we show that moment-angle complexes over $\frac{n}{3}$-neighbourly simplicial complexes are homotopy Golod if and only if they are weakly coherently homotopy Golod, and as such, if and only if they are co-$H$-spaces.

Next, we strengthen a desuspension of the homotopy Golod condition so that it is sufficient for $(D^2,S^1)^K$ to be a co-$H$-space, though possibly no longer necessary.
This is the \emph{coherent homotopy Golod} condition defined below. 
It will be implicit in the proof of Theorem~\ref{TMainStatement2} 
- using Corollary~\ref{CInverse} - that the coherent 
homotopy Golod condition implies the weak coherent homotopy Golod condition.
Also, the following generalisation of the Golod and homotopy Golod conditions - and in turn the coherent homotopy Golod condition - to a relative context will be useful: if $K$ is a subcomplex of $L$, and $I,J\subsetneq [n]$ are disjoint, 
we may ask if the inclusions $\iota_{I,J}\colon\seqm{|K_{I\cup J}|}{}{|L_I\ast L_J|}$ induce a trivial map on cohomology or are nullhomotopic after some suspension.   

In defining the coherent homotopy Golod condition, 
recall that we want the inclusions $\Sigma\iota_{I,J}$ to be nullhomotopic for disjoint non-empty $I,J\subsetneq [n]$, 
and at the same time we want the coherence conditions between these nullhomotopies to hold as outlined in the introduction. 
These properties are encoded as follows.

\begin{figure}[h]
\centering
\includegraphics[scale=0.25]{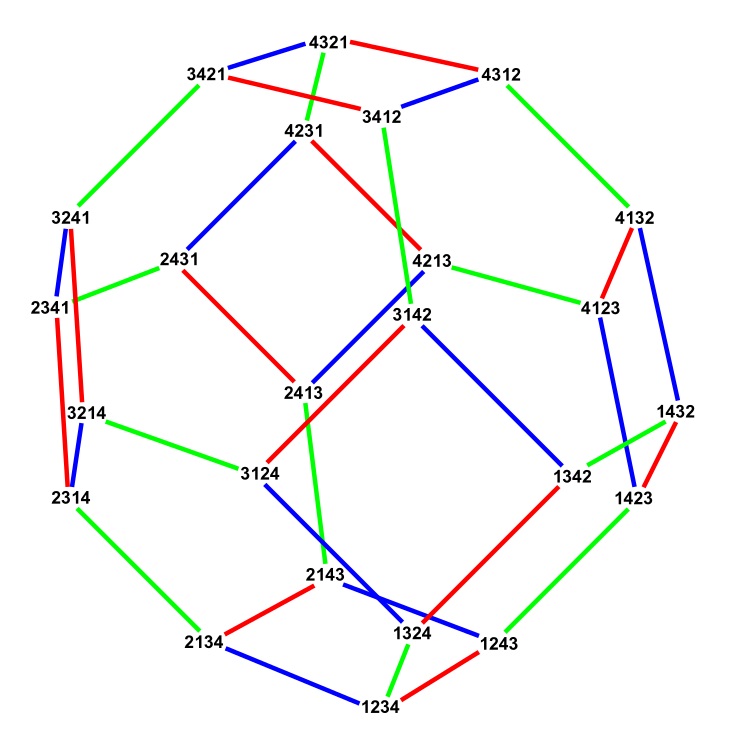}
\caption{\scriptsize{$1$-skeleton of an order $4$ permutohedron}}
\end{figure}

Consider the delta set $\mc K_n=\{\mc F_0,\ldots,\mc F_{n-1}\}$, whose set of $(m-2)$-faces $\mc F_{m-2}$ is the set of all
length $m$ ordered partitions $(I_1,\ldots,I_m)$ of $[n]$, with each $I_j\neq\emptyset$, and whose face maps 
$d_i\wcolon\seqm{\mc F_{m-2}}{}{\mc F_{m-3}}$
are given by $d_i((I_1,\ldots,I_m))=(I_1,\ldots,I_{i-1},I_i\cup I_{i+1},I_{i+2},\ldots,I_m)$ for $1\leq i\leq m-1<n$.
These satisfy the required identity $d_i\circ d_j=d_{j-1}\circ d_i$ when $i<j$. 
Since each ordered partition $\mc S\in\mc F_{m-2}$ is uniquely determined by its $(m-3)$-faces $d_1(\mc S)$,...,$d_{m-1}(\mc S)$, 
then $\mc K_n$ is a simplicial complex of dimension $n-2$. We will see that it is a triangulation of $S^{n-2}$. 
In fact, $\mc K_n$ is the dual of the order $n$ permutohedron.

We will from now on assume $K$ is a subcomplex of $L$ on vertex set $[n]$, $L$ has no ghost vertices, 
though $K$ possibly has less than $n$ vertices if $K\neq L$. If $K=L$, then $K$ also has no ghost vertices. 
We fix $W'=Z_L(D^1,S^0)$, in which case $\hat W'\cong \Sigma|L|$. If $K=L$, then $W'=W$. 

Let $Cone(\Sigma|K|)=[0,1]\times\Sigma|K|/\sim$ under the identifications $(1,x)\sim(t,\ast_{-1})\sim\ast$.

\begin{definition}
The pair $(L,K)$ on vertex set $[n]$ is \emph{coherently homotopy Golod} if $L$ is a single vertex or $K=\emptyset$, 
or (recursively) $(L\backslash\{i\},K\backslash\{i\})$ is coherently homotopy Golod on vertex set $[n]-\{i\}$ for each $i\in [n]$, 
and there is a map
$$
\bar\Psi_{L,K}\wcolon\seqm{|\mc K_n|\times Cone(\Sigma|K|)}{}{\Sigma|\Delta^{n-1}|}
$$
such that for any $\gamma\in |\mc K_n|$, $\bar\Psi_{L,K}(\gamma,\ast)$ is the basepoint $\ast_{-1}\in\Sigma|\Delta^{n-1}|$, and:
\begin{itemize}
\item[(1)] the restriction of $\bar\Psi_{L,K}$ to $\{\gamma\}\times(\{0\}\times\Sigma |K|)$ is the suspended inclusion 
\seqm{\Sigma |K|}{}{\Sigma |L|\subseteq\Sigma|\Delta^{n-1}|}; 
\item[(2)] if $\gamma\in |\mc S|$ for some $\mc S=(I_1,\ldots I_m)\in\mc F_{m-2}$, 
then $\bar\Psi_{L,K}$ maps $\{\gamma\}\times Cone(\Sigma |K|)$ to a subspace of $\Sigma |L_{I_1}\ast\cdots\ast L_{I_m}|$ in $\Sigma|\Delta^{n-1}|$.
\end{itemize}
If $K=L$, we simply say $K$ is \emph{coherently homotopy Golod}, and denote $\bar\Psi_K=\bar\Psi_{K,K}$.
\end{definition}

We will give two more equivalent definitions of coherent homotopy Golodness, which are a bit less directly related to the Golod condition, but which from our standpoint will be more workable. 
Recall that if $\mc S=(s_1,\ldots,s_n)$ is any sequence of real numbers, 
then $[n]_{\mc S}$ is the ordered partition $(I_1,\ldots,I_m)$ of $[n]$, where $I_j=\cset{\ell\in [n]}{s_\ell=s'_j}$,
$s'_j$ is the $j^{th}$ smallest element in the set $\{s_1,\ldots,s_n\}$, and $m=|\{s_1,\ldots,s_n\}|$ is the number is distinct elements in $\mc S$.

\begin{definition}[Second definition of coherent homotopy Golodness]
The pair $(L,K)$ on vertex set $[n]$ is \emph{coherently homotopy Golod} if $L$ is a single vertex or $K=\emptyset$, 
or (recursively) $(L\backslash\{i\},K\backslash\{i\})$ is coherently homotopy Golod  on vertex set $[n]-\{i\}$ for each $i\in [n]$, 
and for some choice of vertex ordering there is a basepoint preserving map
$$
\Psi_{L,K}\wcolon\seqm{\Sigma^n |K|}{}{\Sigma|\Delta^{n-1}|}
$$
such that for any $y=(t_1,\ldots,t_{n-1})\in \frac{[-1,1]^{\times (n-1)}}{\bd([-1,1]^{\times(n-1)})}$ that is not the basepoint, 
$\Psi_{L,K}$ maps the subspace $\{y\}\wedge \Sigma |K|$ of $\Sigma^n|K|$ as follows:
\begin{itemize}
\item[(1)] when $t_1=\cdots=t_{n-1}=0$, the restriction of $\Psi_{L,K}$ to $\{y\}\wedge \Sigma |K|$ is the suspended inclusion 
\seqm{\Sigma |K|}{}{\Sigma |L|\subseteq\Sigma|\Delta^{n-1}|}; 
\item[(2)] letting $\mc S_y=(t_1,\ldots,t_{n-1},0)$ and $(I_1,\ldots,I_m)=[n]_{\mc S_y}$, 
$\Psi_{L,K}$ maps $\{y\}\wedge \Sigma |K|$ to a subspace of $\Sigma |L_{I_1}\ast\cdots\ast L_{I_m}|$ in $\Sigma|\Delta^{n-1}|$.
\end{itemize}
For the base case where $K$ is a single vertex, $\Psi_{L,K}$ can be taken as the constant map.
If $K=L$, we simply say $K$ is \emph{coherently homotopy Golod}, and denote $\Psi_K=\Psi_{K,K}$.
\end{definition} 

\begin{remark}
\label{RInclusion}
If $K$ is on vertex set $J\subseteq [n]$, $|J|=m$, recall the inclusion \seqm{\Sigma |K|}{}{\Sigma|\Delta^{n-1}|} 
is up to homeomorphism the inclusion \seqm{\hat W}{}{(D^1)^{\wedge n}}, given as the restriction of the inclusion
\seqm{(D^1)^{\wedge m}}{}{(D^1)^{\wedge n}} sending $(x_1,\ldots x_m)\mapsto (x'_1,\ldots x'_n)$,
where $x'_i=x_j$ when $i$ is the $j^{th}$ largest element in $J$, otherwise $x'_i=1$. 
\end{remark}

\begin{definition}[Third definition of coherent homotopy Golodness]
Same as the second definition, but with $\Sigma|\Delta^{n-1}|$ replaced with $(D^1)^{\wedge n}$, 
the subspaces $\Sigma |L_{I_1}\ast\cdots\ast L_{I_m}|$,  $\Sigma |L|$, and $\Sigma |K|$ of $\Sigma|\Delta^{n-1}|$
replaced with the subspaces $\hat W'_{I_1}\wedge\cdots\wedge \hat W'_{I_m}$, $\hat W'$, and $\hat W$ of $(D^1)^{\wedge n}$,
and $\Psi_{L,K}$ becoming a map 
$$
\Psi_{L,K}\colon\seqm{\Sigma^{n-1}\hat W}{}{(D^1)^{\wedge n}}.
$$ 
If $K=L$, then $\hat W'_I=\hat W_I\cong \Sigma|K_I|$, and we denote $\Psi_K=\Psi_{K,K}$. 
\end{definition}

\begin{proposition}
\label{PEquivalent}
All three definitions of coherent homotopy Golodness are equivalent.
\end{proposition}

\begin{proof}
Since the homeomorphisms $\Sigma |K|\cong \hat W$ and 
$\Sigma |L_{I_1}\ast\cdots\ast L_{I_m}|\cong \hat W'_{I_1}\wedge\cdots\wedge \hat W'_{I_m}$
that we defined before are restrictions of the homeomorphism $\Sigma|\Delta^{n-1}|\cong (D^1)^{\wedge n}$, 
the equivalence of the second and third definitions follows.

Think of $S^{n-1}$ as the unit $(n-1)$-cube quotient boundary  $[-1,1]^{\times (n-1)}/\bd([-1,1]^{\times(n-1)})$.
We can write $S^{n-1}=\bigcup_{0\leq t\leq 1}U_t$, where $U_t=\cset{(t_1,\ldots,t_{n-1})\in S^{n-1}}{\max\{|t_1|,\ldots,|t_{n-1}|\}=t}$.
When $0<t<1$, $U_t$ is homeomorphic to $S^{n-2}$, while $U_1=\{\ast\}$ and $U_0=\{(0,\ldots,0)\}$, 
where $\ast\in S^{n-1}$ is the basepoint. We construct homeomorphisms
$$
\tau_t\wcolon\seqm{|\mc K_n|}{}{U_t\cong S^{n-2}}
$$
for $0<t<1$ as follows.
Given any face $\mc S=(I_1,\ldots, I_m)\in \mc F_{m-2}$ of $\mc K_n$, its vertices are the length $2$ partitions
$\mc S_i=(I_1\cup\cdots\cup I_i, I_{i+1}\cup\cdots\cup I_m)$ of $[n]$ for $i=1,\ldots,m-1$, 
so we can think of any point $\gamma\in|\mc S|$ as a formal sum 
$$
\gamma=s_1\mc S_1+\cdots+s_{m-1}\mc S_{m-1}
$$ 
such that each $s_i\geq 0$ and $s_1+\cdots+s_{m-1}=1$, and $\gamma$ is on a boundary face $|d_i(\mc S)|\subseteq|\mc S|$ if and only if $s_i=0$, 
in which case the summand $s_i\mc S_i$ vanishes. Given $j\in [n]$, let $j_{\mc S}$ be the integer such that $j\in I_{j_{\mc S}}$,
and let $t_{j,\gamma}=s_0+s_1+\cdots+s_{j_{\mc S}-1}$ where $s_0=0$. Notice at least two of these must be distinct, 
so $\beta_\gamma=\max\{|t_{1,\gamma}-t_{n,\gamma}|,\ldots,|t_{n-1,\gamma}-t_{n,\gamma}|\}>0$. 
Then define $\tau_t$ by 
$$
\tau_t(\gamma) = \paren{\frac{t}{\beta_\gamma}(t_{1,\gamma}-t_{n,\gamma}),\ldots,\frac{t}{\beta_\gamma}(t_{n-1,\gamma}-t_{n,\gamma})}.
$$ 
One can see that $\tau_t$ is well-defined by checking this definition agrees on each boundary face $d_i(\mc S)$. 
If $\gamma$ is in the interior of some $k$-face $|\mc S'|\subseteq|\mc S|$ for $k\leq m$, then $[n]_{\tau_t(\gamma)\times\{0\}}=\mc S'$. 
Then one can check that $\tau_t$ is a homeomorphism with inverse $\tau^{-1}_t\colon\seqm{U_t}{}{|\mc K_n|}$ given by 
$$
\tau^{-1}_t((t_1,\ldots,t_{n-1})) = \paren{\frac{t'_2-t'_1}{t'_k-t'_1}}\mc S'_1+\cdots+\paren{\frac{t'_k-t'_{k-1}}{t'_k-t'_1}}\mc S'_{k-1}
$$ 
where $t'_\ell$ is the $\ell^{th}$ smallest integer in the set $S=\{t_1,\cdots,t_{n-1},0\}$, $k=|S|$,  
and $\mc S'_i$ is the length $2$ partition $(I'_1\cup\cdots\cup I'_i, I'_{i+1}\cup\cdots\cup I'_k)$, given that 
$\mc S'=(I'_1,\ldots, I'_k)$ denotes the partition $[n]_{(t_1,\ldots,t_{n-1},0)}$.

For $t=0$ and $t=1$, we define $\tau_0\colon\seqm{|\mc K_n|}{}{U_0}$ and $\tau_1\colon\seqm{|\mc K_n|}{}{U_1}$ to be the constant maps sending 
all points to $(0,\ldots,0)$ and $\ast$ respectively.

Take the quotient space 
$$
Q=|\mc K_n|\times Cone(\Sigma|K|)/\sim
$$
under the identifications $(\gamma,(0,x))\sim (\gamma',(0,x))\sim(0,x)$ and $(\gamma,\ast)\sim\ast$
for any $\gamma,\gamma'\in |\mc K_n|$ and $x\in \Sigma|K|$. Let
$$
\tau_K\wcolon\seqm{Q}{}{S^{n-1}\wedge \Sigma |K|\cong \Sigma^n |K|}
$$ 
be given by mapping $(\gamma,(t,x))\mapsto (\tau_t(\gamma),x)$. This is a homeomorphism, with inverse
$$
\tau^{-1}_K\wcolon\seqm{S^{n-1}\wedge \Sigma |K|}{}{Q}
$$ 
given by mapping $(0,\ldots,0,x)\mapsto (0,x)$, $\ast\sim(\ast,x)\mapsto \ast$, and 
$(t_1,\ldots,t_{n-1},x)\mapsto (\tau^{-1}_t(t_1,\ldots,t_{n-1}),(t,x))$ when $(t_1,\ldots,t_{n-1})\in U_t$ for some $0<t<1$. 
From the first definition of coherent homotopy Golodness, we see $\bar\Psi_{L,K}$ factors as 
$$
\bar\Psi_{L,K}\wcolon\seqmm{|\mc K_n|\times Cone(\Sigma|K|)}{quotient}{Q}{\hat\Psi_{L,K}}{\Sigma|\Delta^{n-1}|}.
$$
Thus, given $\bar\Psi_{L,K}$ exists as in the first definition, 
we can construct $\Psi_{L,K}$ as in the second definition by taking it as the composite $\hat\Psi_{L,K}\circ\tau^{-1}_K$. 
Conversely, given $\Psi_{L,K}$ exists, we can define $\hat\Psi_{L,K}=\Psi_{L,K}\circ\tau_K$, 
and then take $\bar\Psi_{L,K}$ to be the above composite of the quotient map and $\hat\Psi_{L,K}$. 
Therefore, the first and second definitions are equivalent.

\end{proof}

Notice each parameter $t_i$ in the second and third definition of coherent homotopy Golodness corresponds to the vertex $\{i\}$ in $L$.
But there are $n$ vertices and only $n-1$ parameters $t_i$, the vertex $\{n\}$ being treated exceptionally here. 
Then both the second and third definitions depend on the choice of ordering of vertices in $K$ and $L$ by elements in $[n]$.
The first definition of coherent homotopy Golodness is however independent of vertex ordering, 
meaning it holds for all orderings if it holds for one. Moreover, in the proof of Proposition~\ref{PEquivalent}
the argument that the first definition implies the second and third works for any choice of ordering of vertices.
We can thus go from the second or third definitions to the first definition, change the ordering, 
and then come back again with a new map $\Psi_{L,K}$ corresponding to the new ordering. Therefore:

\begin{corollary}
\label{COrdering}
If the second or third definitions of coherent homotopy Golodness hold for one vertex ordering, then they hold for all vertex orderings.~$\qqed$
\end{corollary}

We will use the second and third definitions of coherent homotopy Golodness from now on. 

\begin{theorem}
\label{TMainStatement2}
If $K$ is coherently homotopy Golod, then $(D^2,S^1)^K$ is a co-$H$-space.
\end{theorem} 

\begin{proof}
Induct on number of vertices. The base case is trivial. Assume the statement holds for simplicial complexes on less than $n$ vertices.
Since by definition, each $K_I$ is coherently homotopy Golod when $K$ is, 
by our inductive assumption $W^1_I$ is a co-$H$-space for $|I|<n$. Then using Proposition~\ref{PSplitting}, 
the quotient map \seqm{W^1_I}{}{\hat W^1_I} has a right homotopy inverse for $|I|<n$, 
and to show $W^1$ is a co-$H$-space, it remains to show that \seqm{W^1}{}{\hat W^1} has a right homotopy inverse.

Given our map $\Psi_K\colon\seqm{\Sigma^{n-1}\hat W}{}{(D^1)^{\wedge n}}$, 
define a basepoint preserving map $g\colon\seqm{\Sigma^{n-1}\hat W}{}{\mc D_{\mc A}(W)}$ by mapping a point $\sigma=(t_1,\ldots,t_{n-1},x)$ 
by
$$
g(\sigma)=(t_1,\ldots,t_{n-1},0;\Psi_K(\sigma)).
$$
The composite of $g$ and the inclusion $q\colon\seqm{\mc D_{\mc A}(W)}{\iota}{\mb C(\iota_{\mc A,n})}$
is homotopic to the homotopy equivalence \seqm{\Sigma^{n-1}\hat W}{\phi}{\mb C(\iota_{\mc A,n})}
from Lemma~\ref{LCofib} as follows. Let $\beta=\max\{|t_1|,\ldots,|t_{n-1}|,0\}$.
Since no more than $n-1$ particles in $g(\sigma)$ collide at a point  when $\beta>0$, 
$g(\sigma)\subseteq\mc D_{\mc A}(W_{n-1})$ when $\beta>0$, 
so we can homotope $q\circ g(\sigma)=(0,g(\sigma))$ to $(\beta,g(\sigma))$ 
inside $\mb C(\iota_{\mc A,n})$ preserving the basepoint by homotoping $0$ to $\beta$ linearly.
Then we homotope $(\beta,g(\sigma))$ to $\phi(\sigma)=(\beta,(t_1,\ldots,t_{n-1},0;x))$ via the basepoint preserving homotopy
$$
H_t(\sigma)=
\begin{cases}
\ast & \mbox{ if }\beta=1;\\
(\beta,(t_1,\ldots,t_{n-1},0;\Psi_K((1-t)t_1,\ldots,(1-t)t_{n-1},x))) & \mbox{ if }\beta<1.
\end{cases}
$$
Notice $H_t(\sigma)$ is well-defined when $(t_1,\ldots,t_{n-1})$ is on the boundary of $[-1,1]^{\times(n-1)}$, 
since then $\beta=1$, and is continuous as $(t_1,\ldots,t_{n-1})$ approaches the boundary of $[-1,1]^{\times(n-1)}$, 
since then $\beta$ approaches $1$, so $(\beta,(t_1,\ldots,t_{n-1},0;\Psi_K((1-t)t_1,\ldots,(1-t)t_{n-1},x)))$
approaches the basepoint in $\mb C(\iota_{\mc A,n})$.
Thus, since $\phi$ is a homotopy equivalence, the composite of $g$ and \seqm{\mc D_{\mc A}(W)}{q}{\mb C(\iota_{\mc A,n})}
is a homotopy equivalence, so $q$ has a right homotopy inverse, and then by Corollary~\ref{CInverse}, so does \seqm{W^1}{}{\hat W^1}. 
\end{proof}

Next we show that two general classes of Golod complexes - whose moment angle complexes are co-$H$-spaces - are in fact coherently homotopy Golod. Before getting to these, the next example will be useful.

\begin{example}
\label{EDisjointVertex}
Suppose $K$ with no ghost vertices is coherently homotopy Golod on vertex set $[n]$, 
and let $L=K\cup \{n+1\}$ be $K$ with a disjoint vertex $\{n+1\}$ added. 
Then the pair $(L,K)$ is coherently homotopy Golod on vertex set $[n+1]$ as follows. 

Assume this is true for complexes $K$ with less than $n$ vertices. The base case $K=\emptyset$ holds by definition.
Since $L\backslash\{i\}=K\backslash\{i\}\cup\{n+1\}$ when $i\neq n+1$, then by induction each 
$(L\backslash\{i\},K\backslash\{i\})$ is coherently homotopy Golod on vertex set $[n+1]-\{i\}$, 
while $K$ being coherently homotopy Golod implies by definition that $(L\backslash\{n+1\},K\backslash\{n+1\})=(K,K)$ is.  
Thus we are left to show the existence of the map $\Psi_{L,K}$.
 
As before, $W'=Z_L(D^1,S^0)$ and $W=Z_K(D^1,S^0)$.  
Given our map $\Psi_K\colon\seqm{\Sigma^{n-1}\hat W}{}{(D^1)^{\wedge n}}$ from the third definition of coherent homotopy Golodness of $K$,
define $\Psi_{L,K}\colon\seqm{\Sigma^n\hat W}{}{(D^1)^{\wedge (n+1)}}$ by mapping the basepoint to the basepoint, and a non-basepoint 
$\sigma=(t_1,\ldots,t_{n-1},s,(x_1,\ldots,x_n))\in\Sigma^n\hat W$ to 
$$
\Psi_{L,K}(\sigma)=(x'_1,\ldots,x'_n,y),
$$ 
where $(x'_1,\ldots,x'_n)=\Psi_K(t_1,\ldots,t_{n-1},(x_1,\ldots,x_n))$, 
$$
y=
\begin{cases}
1 & \mbox{if }\alpha\leq s\leq\beta;\\
2\paren{\frac{1-s}{1-\beta}}-1 & \mbox{if }\beta<s<1;\\
2\paren{\frac{1+s}{1+\alpha}}-1 & \mbox{if }-1<s<\alpha;\\
-1 & \mbox{if }|s|=1,
\end{cases}
$$
and $\alpha=\min\{t_1,\ldots,t_{n-1},0\}$ and $\beta=\max\{t_1,\ldots,t_{n-1},0\}$. 
Notice there is a discontinuity here for $y$ as both $s$ and $\beta$ approach $1$, or $s$ and $\alpha$ approach $-1$. 
But this does affect the continuity of $\Psi_{L,K}$ since ($\Psi_K$ being basepoint preserving) $(x'_1,\ldots,x'_n)$ approaches the basepoint 
as $\alpha$ approaches $-1$ or $\beta$ approaches $1$, so $\Psi_{L,K}(\sigma)$ approaches the basepoint here. 
Likewise, $\Psi_{L,K}(\sigma)$ approaches the basepoint if 
(1) some $x_i$ approaches $-1$ (since some $x'_j$ approaches $-1$, since $\Psi_K$ is basepoint preserving);
or (2) $s$ approaches $1$ or $-1$ and $\beta$ or $\alpha$ respectively do not (since $y$ approaches $-1$).
In any case, $\Psi_{L,K}(\sigma)$ approaches the basepoint as $\sigma$ approaches the basepoint.

Since $\Psi_K$ satisfies the first condition of coherent homotopy Golodness,
$\Psi_{L,K}(\sigma)=(x'_1,\ldots,x'_n,y)=(x_1,\ldots,x_n,1)$ when the $t_i$'s and $s$ are zero, 
so $\Psi_{L,K}$ restricts to the inclusion \seqm{\hat W}{}{\hat W'} here, 
thus also satisfies the first condition.

Note $\hat W'_I=\hat W_I$ if $\{n+1\}\nin I$.
Let $\mc S=(t_1,\ldots,t_{n-1},0)$ and $(I_1,\ldots,I_m)=[n]_{\mc S}$. 
Since $\Psi_K$ satisfies the second condition of coherent homotopy Golodness,
$(x'_1,\ldots,x'_n)$ is in $A=\hat W_{I_1}\wedge\cdots\wedge\hat W_{I_m}$.
Let $\mc S'=(t_1,\ldots,t_{n-1},s,0)$ and $(I'_1,\ldots,I'_{m'})=[n+1]_{\mc S}$.
If $s=0$ or $s=t_i$ for some $i$, then $m=m'$,
$I'_k=I_k\cup\{n+1\}$ for some $k$, $I'_i=I_i$ for $i\neq k$, and $y=1$ (since then $\alpha\leq s\leq\beta$). 
In any case, each $I_j\subseteq I'_j$.
Thus $\Psi_{L,K}(\sigma)=(x'_1,\ldots,x'_n,1)$ is in the subspace $A$ of 
$A'=\hat W'_{I_1}\wedge\cdots\wedge\hat W'_{I_m}$
given on each factor by the inclusions \seqm{\hat W_{I_i}}{}{\hat W'_{I'_i}} as described in Remark~\ref{RInclusion}, 
these being equalities when $i\neq k$. 
On the other hand, if $s\neq 0$ and $s\neq t_i$ for each $i$, then $m'=m+1$, $I'_k=\{n+1\}$ for some $k$, 
and $I'_i=I_i$ and $I'_j=I_{j-1}$ for $i<k<j$, 
so $\hat W_{I_k}=D^1$, $\hat W_{I_i}=\hat W_{I'_i}$, and $\hat W_{I_{j-1}}=\hat W_{I'_j}$, for $i<k<j$. 
Then $\Psi_{L,K}(\sigma)$ is in the (homeomorphic) subspace $A\wedge D^1$ of $A'$.
Thus, we see $\Psi_{L,K}$ satisfies the second condition of coherent homotopy Golodness.  
\end{example}

\begin{example}
Recall from~\cite{arXiv:1306.6221} 
that a complex $K$ on vertex set $[n]$ (without ghost vertices) is \emph{extractible} if either $K\backslash\{i\}$ is a simplex for some $i\in [n]$, 
or else (recursively) $K\backslash\{i\}$ is extractible for each $i\in [n]$ and the wedge sum of inclusions 
$\seqm{\bigvee_{i\in [n]}\Sigma|K\backslash\{i\}|}{}{\Sigma |K|}$ has a right homotopy inverse 
$$
s\wcolon\seqm{\Sigma |K|}{}{\bigvee_{i\in [n]}\Sigma|K\backslash\{i\}|}.
$$ 
We show that $K$ is coherently homotopy Golod by inducting on number of vertices.
Suppose any extractible complex on less than $n$ vertices is coherently homotopy Golod, the base case being trivial.
Then since each $K\backslash\{i\}$ is extractible when $K$ is, each $K\backslash\{i\}$ is coherently homotopy Golod on vertex set $[n]-\{i\}$ by induction.
Thus, to show $K$ is coherently homotopy Golod, we are left only to construct the map $\Psi_K\colon\seqm{\Sigma^{n-1}\hat W}{}{(D^1)^{\wedge n}}$
(or $\Psi_K\colon\seqm{\Sigma^n|K|}{}{\Sigma|\Delta^{n-1}|}$ from the second definition).
 
First consider the case where $K\backslash\{i\}$ is a simplex for some $i\in [n]$. Without loss of generality, assume $i=n$.
Let $W=(D^1,S^0)^K$ and define $\Psi_K\colon\seqm{\Sigma^{n-1}\hat W}{}{(D^1)^{\wedge n}}$ by mapping 
$\sigma=(t_1,\ldots,t_{n-1},(x_1,\ldots,x_n))$ to
$$
\Psi_K(\sigma)=\paren{x'_1,\ldots,x'_{n-1},x_n}.
$$
where $x'_i=(1-|t_i|)x_i-|t_i|$. The last coordinate $x_n$ that is unchanged corresponds to the vertex $\{n\}$.
This clearly satisfies the first condition of coherent homotopy Golodness.
One can see that it satisfies the second condition precisely because $K_I$ is a simplex for any $I\subseteq[n]-\{n\}$, 
that is, $\hat W_{I}=(D^1)^{\wedge|I|}$.
So for the second condition to hold, each of the first $n-1$ coordinates $x'_i$ of $\Psi_K(\sigma)$ still has the freedom 
to take on any value in $D^1$, as long as the corresponding $t_i$ is non-zero.

For the general case, let $K^i=K\backslash\{i\}$ and take the disjoint union $L^i=(K^i)\cup \{i\}$. 
Since by induction each $K^i$ is coherently homotopy Golod on vertex set $[n]-\{i\}$, 
by Example~\ref{EDisjointVertex} the pair $(L^i,K^i)$ is coherently homotopy Golod on vertex set $[n]$.
Let $\Psi_{L^i,K^i}\colon\seqm{\Sigma^n|K^i|}{}{\Sigma|\Delta^{n-1}|}$ be the associated map. 
By Corollary~\ref{COrdering} we can assume each $\Psi_{L^i,K^i}$ has been taken so that $\{n\}$ is the exceptional vertex, 
and each of the first $n-1$ vertices $\{i\}\in L^i$ correspond to the parameter $t_i$ in the second definition of coherent homotopy Golodness.
Consider the map
$$
\Psi'_K\wcolon
\seqmmm{\Sigma^n |K|}{\Sigma^{n-1}s}{\cvee{i\in [n]}{}{\Sigma^n|K^i|}}{\vee\Psi_{L^i,K^i}}
{\cvee{i\in [n]}{}{\Sigma|\Delta^{n-1}|}}{}{\Sigma|\Delta^{n-1}|}
$$  
where the last map is the fold map.
Notice that $\Psi'_K$ satisfies the second condition of coherent homotopy Golodness since each $\Psi_{L^i,K^i}$ does,
since the first map $\Sigma^{n-1}s$ and second map are $(n-1)$-fold suspensions (thus, do not change the first $n-1$ coordinates $t_i$),
and since each $|L^i_{I_1}\ast\cdots\ast L^i_{I_m}|$ is a subspace of $|K_{I_1}\ast\cdots\ast K_{I_m}|$ in $|\Delta^{n-1}|$ 
(since $K$ has no ghost vertices). However, $\Psi'_K$ does not necessarily satisfy the first condition. 
Instead, as each $\Psi_{L^i,K^i}$ satisfies the first condition of coherent homotopy Golodness and $\Sigma^{n-1}s$ is an $(n-1)$-fold suspension,
$\Psi'_K$ maps the subspace $\{(0,\ldots,0)\}\wedge \Sigma |K|$ to $\Sigma |K|$ in $\Sigma|\Delta^{n-1}|$ 
via the composite \seqmm{\Sigma |K|}{s}{\bigvee_i\Sigma|K\backslash\{i\}|}{}{\Sigma |K|}, 
thus the restriction of $\Psi'_K$ to this subspace is homotopic to the identity.
Denote this homotopy by $H\colon\seqm{\Sigma|K|\times [0,1]}{}{\Sigma |K|}$. Pick some $0<\varepsilon<1$. Let
$$
\chi(t)=
\begin{cases}
0 & \mbox{if }|t|\leq\varepsilon;\\
\frac{t-\varepsilon}{1-\varepsilon} & \mbox{if }t>\varepsilon;\\
\frac{t+\varepsilon}{1-\varepsilon} & \mbox{if }t<-\varepsilon,
\end{cases}
$$
and for any $\sigma=(t_1,\ldots,t_{n-1},x)\in\Sigma^n|K|$, where $x\in\Sigma|K|$, let $\beta=\max\{|t_1|,\ldots,|t_{n-1}|,0\}$,
and
$$
t_\sigma=
\begin{cases}
\frac{\varepsilon-\beta}{\varepsilon}  & \mbox{if }\beta\leq\varepsilon;\\
0 & \mbox{if }\beta>\varepsilon.
\end{cases}
$$
Note $|\chi(t)|$ approaches $1$ as $|t|$ approaches $1$. Define $\Psi_K\colon\seqm{\Sigma^n|K|}{}{\Sigma|\Delta^{n-1}|}$ by
$$
\Psi_K(\sigma)=\Psi'_K(\chi(t_1),\ldots,\chi(t_n),H_{t_\sigma}(x)).
$$    
Now $\Psi_K$ satisfies the first condition of coherent homotopy Golodness. 
Since $\chi(t_i)\leq\chi(t_j)$ when $t_i<t_j$, $\chi(t_i)=\chi(t_j)$ when $t_i=t_j$, 
and $\Psi_K$ is defined in terms of $H_{t_\sigma}$ (which maps to $\Sigma|K|\subseteq\Sigma|\Delta^{n-1}|$) 
and $\Psi'_K$ (which satisfies the second condition of coherent homotopy Golodness), 
we see $\Psi_K$ also satisfies the second condition.
\end{example}

\begin{example}
Define a simplicial complex $K$ on vertex set $[n]$ to be $m$-\emph{neighbourly} if every subset of at most $m$ vertices in $[n]$ is a face in $K$. 
We say $K$ is \emph{neighbourly} if it is $\floor{\frac{n}{2}}$-neighbourly.

Suppose $K$ is neighbourly.
Since $|I|\leq\floor{\frac{n}{2}}$ or $|J|\leq\floor{\frac{n}{2}}$ for any disjoint non-empty $I,J\subsetneq [n]$, 
then at least one of $|K_I|$ or $|K_J|$ is a simplex, implying $|K_I\ast K_J|\simeq\Sigma |K_I|\wedge |K_J|$ is contractible. 
Therefore each inclusion \seqm{\Sigma |K|}{\Sigma\iota_{I,J}}{\Sigma |K_I\ast K_J|} is nullhomotopic, implying $K$ is Golod. 
We show $K$ is coherently homotopy Golod by inducting on number of vertices as follows.

Suppose any neighbourly complex on less than $n$ vertices is coherently homotopy Golod. 
Since each $K\backslash\{i\}$ is neighbourly when $K$ is, $K\backslash\{i\}$ is coherently homotopy Golod by induction. 
Now define 
$$
\Psi_K\colon\seqm{\Sigma^{n-1}\hat W}{}{(D^1)^{\wedge n}}
$$ 
by mapping the basepoint to the basepoint, and a non-basepoint $\sigma=(t_1,\ldots,t_{n-1},(x_1,\ldots,x_n))$ to 
$$
\Psi_K(\sigma)=\paren{f_1(x_1),\ldots,f_n(x_n)},
$$
where $\beta=\max\{|t_1|,\ldots,|t_{n-1}|,0\}$,
$$
f_i(x_i)=
\begin{cases}
(1-2\alpha_i\beta)(x_i+1)-1 & \mbox{if }0\leq\alpha_i<\frac{1}{2};\\
(1-\beta)(x_i+1)-1 & \mbox{if }\alpha_i\geq\frac{1}{2},
\end{cases}
$$ 
and (letting $t_n=0$)
$$
\alpha_i=\min\cset{\csum{j\in S}{}{|t_i-t_j|}}{S\subseteq [n]-\{i\}\mbox{ and }|S|=\floor{\frac{n}{2}}}.
$$
We must show $\Psi_K$ is continuous. Any possible discontinuity happens when $\beta$ approaches $1$,
in other words, when $(t_1,\ldots,t_{n-1})$ approaches the boundary of $[-1,1]^{\times (n-1)}$,
or when $x=(x_1,\ldots,x_n)$ approaches the basepoint. To see that there is no issue here, 
notice $S\cap S'$ is non-empty for any $S\subseteq [n]-\{i\}$ and $S'\subseteq [n]-\{n\}$ such that $i\neq n$ 
and $|S|=|S'|=\floor{\frac{n}{2}}$, implying $\max\{\alpha_n,\alpha_i\}\geq \frac{|t_i-t_n|}{2}$.
So if $|t_i|$ approaches $1$ for some $i\neq n$, since $t_n$ is always $0$, 
at least one of $\alpha_n$ or $\alpha_i$ approaches some value $L\geq\frac{1}{2}$,
so one of $f_i(x_i)$ or $f_n(x_n)$ approaches $-1$. Also, since $f(x_i)$ approaches $-1$ as $x_i$ does,
$\Psi_K(\sigma)$ approaches the basepoint when $x$ does. 
In any case, $\Psi_K(\sigma)$ approaches the basepoint as $\sigma$ approaches the basepoint.

Clearly $\Psi_K$ satisfies the first condition of coherent homotopy Golodness. 
Since $\hat W_I= (D^1)^{\wedge |I|}$ whenever $|I|\leq \floor{\frac{n}{2}}$, 
in order for $\Psi_K$ satisfy the second coherent homotopy Golod condition,
for each $i$, $f_i(x_i)$ has freedom to taken on any value in $[-1,1]$ whenever there are no more than 
$\floor{\frac{n}{2}}$ values of $j$ for which $t_j=t_i$, so defining $f_i$ as we did presents no difficulties here.
On the other hand, when there are strictly more than $\floor{\frac{n}{2}}$ values of $j$ for which $t_j=t_i$,
we have $f_i(x_i)=x_i$ and each $f_j(x_j)=x_j$ (since $\alpha_i=\alpha_j=0$).
So because $x=(x_1,\ldots,x_n)\in \hat W$, the restriction of $x$ to the coordinates corresponding to $x_i$ and each $x_j$
is an element in some $\hat W_I$ (this does not give a continuous map \seqm{\hat W}{}{\hat W_I}), 
so there is no issue here either. 
Thus $\Psi_K$ satisfies the second condition, and $K$ is coherently homotopy Golod, therefore a co-$H$-space. 

We note that another argument showing $\floor{\frac{n}{2}}$-neighbourly moment-angle complexes are 
co-$H$-spaces appears in~\cite{BebenGrbic2}, this time using the Freudenthal suspension theorem. 

\end{example}

\begin{example}
Let $K=(\bd\Delta^2\ast\bd\Delta^1)\cup_{\bd\Delta^1}\Delta^1$ on vertex set $[5]$. 
Since every pair of distinct vertices in $K$ is connected by an edge, 
$K$ is neighbourly, and therefore coherently homotopy Golod. 
But $H_2(|K|)\cong \mb Z$, while $H_2(|K\backslash\{i\}|)=0$ for each $i\in [5]$,
so $K$ cannot be extractible. 
\end{example}

\begin{question}
Do any of the following implications hold: (1) $K$ Golod $\Rightarrow$ $K$ homotopy Golod;  
(2) $K$ homotopy Golod $\Rightarrow$ $K$ weakly coherently homotopy Golod;
(3) $K$ weakly coherently homotopy Golod $\Rightarrow$ $K$ coherently homotopy Golod;
\end{question}

Conjecture~\ref{CGolodConj} is false if (1) is false, and true if both (1) and (2) are true.

\bibliographystyle{amsplain}
\bibliography{mybibliography}

\end{document}